\documentclass[final,11pt]{article}
\usepackage{amsthm}
\usepackage{cite}
\usepackage{amsmath,amssymb,amsfonts}
\usepackage[dvipsnames,usenames]{color}

\usepackage{hyperref}
\usepackage[bitstream-charter]{mathdesign}
\usepackage[utf8]{inputenc}
\usepackage{tikz}
\usepackage{caption}
\usepackage{subcaption}
\usepackage{enumerate}
\usepackage{geometry}
\usepackage[affil-it]{authblk}
\usepackage{tikz-cd}
\usepackage{titling}

\newtheorem{thm}{Theorem}[section]
\newtheorem{prop}[thm]{Proposition}%[section]
\newtheorem{lem}[thm]{Lemma}%[section]
\newtheorem{quest}{Question}

\newtheorem{cor}[thm]{Corollary}%[section]
\newtheorem{rem}[thm]{Remark}%[section]
\newtheorem{obs}[thm]{Observation}
\newtheorem{conj}{Conjecture}%[section]
\newtheorem{thmintro}{Result}%[section]

\newtheorem{case}{Case}

\newcommand{\covectors}{\ensuremath{\mathcal{L}}}

\def\B{\,\square \,}

\begin{document}
\thanksmarkseries{arabic}
\title{Corners and simpliciality in oriented matroids and  partial cubes}
\author{Kolja Knauer\thanks{Laboratoire d'Informatique et Systh\`emes, Aix-Marseille Universit\'e, Marseille, France \& Departament de Matem\`atiques i Inform\`atica,
Universitat de Barcelona (UB), Barcelona, Spain,  \texttt{kolja.knauer@lis-lab.fr}} \and Tilen Marc\thanks{Faculty of Mathematics and Physics, Ljubljana, Slovenia, Institute of Mathematics, Physics and Mechanics, Ljubljana, Slovenia \& XLAB d.o.o., Ljubljana, Slovenia, \texttt{tilen.marc@fmf.uni-lj.si}}}
% 
% {
% Kolja Knauer \thanks{Laboratoire d'Informatique et Systh\`emes, Aix-Marseille Universit\'e, Marseille, France and Departament de Matem\`atiques i Inform\`atica,
% Universitat de Barcelona (UB), Barcelona, Spain, \texttt{kolja.knauer@lis-lab.fr}} 
% 
% \and
% 
% Tilen Marc\thanks{Faculty of Mathematics and Physics, Ljubljana, Slovenia, Institute of Mathematics, Physics and Mechanics, Ljubljana, Slovenia  XLAB d.o.o., Ljubljana, Slovenia, \texttt{tilen.marc@fmf.uni-lj.si}}
% }

% %\affil{Institute of Mathematics, Physics, and Mechanics, Jadranska 19, 1000 Ljubljana, Slovenia}

\maketitle

\begin{abstract} 
Building on a recent characterization of tope graphs of Complexes of Oriented Matroids (COMs), we tackle and generalize several classical problems in Oriented Matroids (OMs), Lopsided Sets (aka ample set systems), and partial cubes via Metric Graph Theory. %These questions are related to the notion of simpliciality of topes in Oriented Matroids and the concept of corners in Lopsided Sets arising from computational learning theory.

Our first main result is that every element of an OM from a class introduced by Mandel is incident to a simplicial tope, i.e, such OMs contain no mutation-free elements. This allows us to refute a conjecture of Mandel from 1983, that would have implied the famous Las Vergnas' simplex conjecture. Further, we show that the mutation graph of uniform OMs of order at most 9 are connected, thus confirming a stronger conjecture of Cordovil-Las Vergnas in this setting.

The second main contribution is the introduction of corners of COMs as a natural generalization of corners in Lopsided Sets. Generalizing results of Bandelt and Chepoi, Tracy Hall, and Chalopin et al. we prove that realizable COMs, rank 2 COMs, as well as hypercellular graphs admit corner peelings. 
Using this, we confirm Las Vergnas' conjecture for antipodal partial cubes (a class much lager than OMs) of small rank or isometric dimension.

% %On the way we introduce the notion of cocircuit graphs for pure COMs and disprove a conjecture about realizability in COMs of Bandelt et al. 
% 
% Finally, we study extensions of Las Vergnas' simplex conjecture in low rank and order. %We first consider antipodal partial cubes -- a vast generalization of oriented matroids also known as acycloids. 
% We prove Las Vergnas' conjecture for acycloids (aka antipodal partial cubes) of rank~3 and for acycloids of order at most~7. Moreover, we confirm a conjecture of Cordovil-Las Vergnas about the connectivity of the mutation graph of Uniform Oriented Matroids for ground sets of order at most 9. %The latter two results are based on the exhaustive generation of acycloids and uniform oriented matroids of given order, respectively.
\end{abstract}

\section{Introduction}\label{sec:intro}
 
%This class has a crucial role in understanding structures emerging in high-dimensional hypercubes and more generally the Cartesian products of graphs. 

The \emph{hypercube} $Q_n$ of dimension $n$ is the graph whose vertex set is $\{+, -\}^n$ where two vertices are adjacent if they differ in exactly one coordinate.
A graph $G$ is called a \emph{partial cube} if $G$ is an isometric subgraph of a hypercube $Q_n$, i.e., $d_G(u,v)=d_{Q_n}(u,v)$ for all $u,v\in G$. In this case, the minimum such $n$ is the \emph{isometric dimension} of $G$.

An Oriented Matroid (OM) is a mathematical structure that abstracts the properties of hyperplane arrangements, directed graphs, vector arrangements and others. It has been studied extensively in combinatorial, as well as in topological and algebraic settings. For graph theorists it might be attractive to consider the \emph{tope graph} of an OM. It can easily be shown that the tope graph is a partial cube and determines a simple OM uniquely up to isomorphism~\cite{Bjo-90}, hence  no information is lost in this approach.
%In fact, this view highlights a surprising connection between OMs and many graph classes appearing naturally in many places, e.g. diagrams of distributive lattices, antimatroids, median graphs, skeleta of CAT(0) cube complexes, linear extensions graphs of posets, region graphs of pseudoline arrangements and hyperplane arrangements, and Pasch graphs. 
%While all the above families can be seen as subfamilies of partial cubes, they have been subject to much more research than partial cubes in general. However, the point of view of partial cubes allows an intuitive approach to these structures. 

In this paper based on recent graph theoretical characterization~\cite{Kna-17}, we analyse OMs and related structures through their tope graphs. In particular, in the first part of the paper this view leads us to new results concerning Las Vergnas simplex conjecture and mutations. While the approach to this section is graph theoretical, we try to present it in a way that a reader more familiar with the standard approach to OMs can understand it. In the second part of the paper we explore how these problems and ideas can be extended beyond OMs into two directions. First, we consider Complexes of Oriented Matroids (COMs), which are a recent generalization of OMs, who share the property that their tope graph is a partial cube and determines a simple COM uniquely up to isomorphism~\cite{Ban-18}. Second, we consider \emph{antipodal} partial cubes, i.e., those isometric subgraphs $G$ of $Q_n$ such that for every $v\in G$ also its \emph{antipode} $-v$, where all signs are reversed, is a vertex of $G$. Note that if $G$ is isometrically embedded in $Q_n$ with minimal $n$, the embedding is unique up to an isomorphism ~\cite[Chapter 5]{Ovc-11}, so an antipode is well defined even if the embedding is not given. Antipodal partial cubes are the tope graphs of so-called \emph{acycloids}~\cite{Fuk-93}. Both of these generalizations are natural since $G$ is the tope graph of an OM if and only it is antipodal and the tope graph of a COM~\cite{Kna-17}.

So, what is a tope graph? One of the standard ways to define an OM, a COM (or an acycloid) is as a pair $\mathcal{M}=(E,\mathcal{L})$ of a finite \emph{ground set} $E$ and a set of \emph{covectors} $\mathcal{L} \subset \{+,-,0\}^E$ with certain properties, see Section \ref{sec:prel}. One obtains the \emph{tope graph}, by considering the subgraph induced in the hypercube by the \emph{topes}, i.e., covectors without $0$-entries. Considering an OM as its tope graph, we will denote it just as $G$. All OM notions have an equivalent description in terms of the tope graph, e.g., the order of $E$ is the isometric dimension of $G$. The \emph{rank} of an OM $\mathcal{M}=(E,\mathcal{L})$, classically defined as the length of the maximal chains in $\mathcal{L}$ with respect to the product order induced by $0<+, 0<-$, can be read	 of its tope graph $G$ as well, see Section~\ref{sec:prel}.

If $G$ is an OM, then a vertex $v\in G$ is called  \emph{simplicial} if $\deg(v)=r(G)$.  Simplicial vertices correspond to simplicial topes and it is a well known fact that the degree of each vertex (tope) must be at least $r(G)$. The well-known simplex conjecture of Las Vergnas~\cite{lv-80} in terms of tope graphs of OMs reads as follows:

\begin{conj}[Las Vergnas]\label{conj:lasvergnas}
 Every OM has a simplicial tope.
\end{conj}
The conjecture is motivated by the fact that it holds for all OMs that are \emph{realizable} (by hyperplanes in Euclidean space)~\cite{Sha-79}.
The largest class known to satisfy Conjecture~\ref{conj:lasvergnas} was found in~\cite[Theorem 7]{Man-82}. We call that class \emph{Mandel} here and consider it in depth in Section \ref{sec:mandel}. Realizable OMs and OMs of rank at most $3$ are \emph{Euclidean} and the latter are Mandel, but the class is larger.
Indeed, Mandel~\cite[Conjecture 8]{Man-82} even conjectured the following as a ``wishful thinking statement'', since by the above it would imply the conjecture of Las Vergnas. 

\begin{conj}[Mandel]\label{conj:mandel}
 Every OM is Mandel\footnote{
Note that it was us and not Mandel who baptized this class of OMs, we merely 
want to combine eponymous homage and correct attribution.}.
\end{conj}
%, leading to this presentation of the conjecture.

Let us now consider some strengthenings of Las Vergnas' conjecture. Recall that $G$ is a partial cube, and hence its edges are partitioned into so-called $\Theta$-classes corresponding to $E$, i.e., the coordinates of the hypercube $G$ is embedded in. We say that such $G$ is \emph{$\Theta$-Las Vergnas}, if every $\Theta$-class of $G$ contains an edge incident to a simplicial vertex. In the language of OMs this means that $G$ has no mutation-free elements. It is known that rank~$3$ OMs are $\Theta$-Las Vergnas~\cite{Lev-26}. 
This brings us to the first main result of this paper. In Theorem~\ref{thm:simplical_on_theta} we extend the class of $\Theta$-Las Vergnas OMs significantly, by showing that Mandel OMs are $\Theta$-Las Vergnas.
On the other hand, OMs that are not $\Theta$-Las Vergnas have been discovered~\cite{Ric-93,Bok-01,Tra-04} . See Figure~\ref{fig:tracyhall} for an illustration of one of them.
Together with Theorem~\ref{thm:simplical_on_theta} this disproves Mandel's conjecture (Corollary~\ref{cor:Mandelwaswrong}).

\begin{thmintro}\label{thm:intro_theta_lasvergnas}
Mandel OMs are $\Theta$-Las Vergnas, hence Mandel's Conjecture fails.%but not all OMs are $\Theta$-Las Vergnas disproving Mandel's conjecture.
\end{thmintro}

\begin{figure}[ht]
\centering
\includegraphics[width=.6\textwidth]{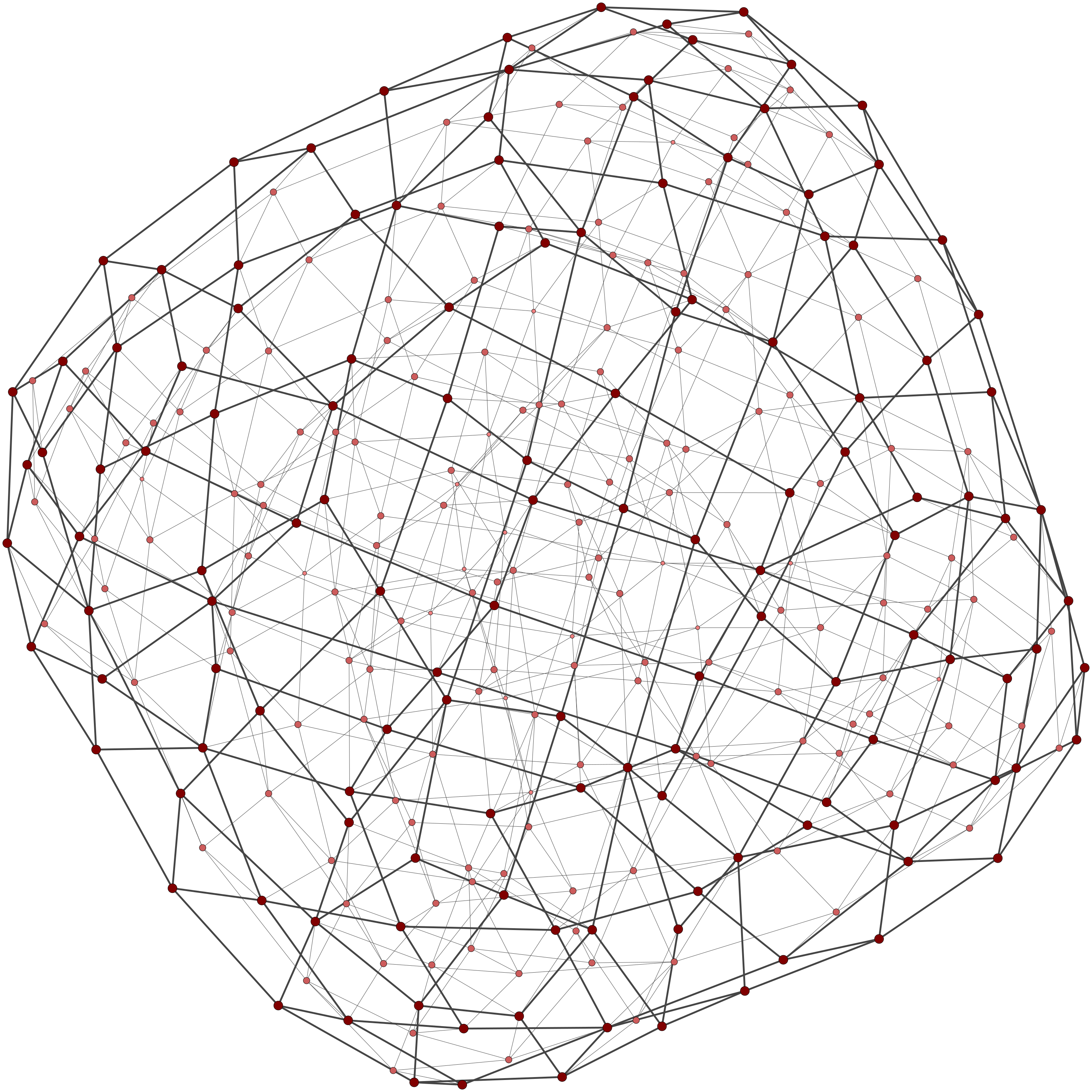}
\caption{A halfspace $E_e^+$ of a non-Mandel OM $G$, where $E_e$ is the $\Theta$-class witnessing that $G$ is not $\Theta$-Las Vergnas. The bold subgraph is an OM of rank~3 on the vertices incident with $E_e$ in $G$.}
\label{fig:tracyhall}
\end{figure}

The next is a strengthening of Las Vergnas' conjecture for uniform Oriented Matroids (UOMs), see Section \ref{sec:prel} for a definition. 
A well-known operation that one can apply to a simplicial vertex of an UOM is a \emph{mutation} - transforming one UOM to another one, see Section~\ref{sec:mut}.
One can consider the \emph{mutation graph}  $\mathcal{G}^{n,r}$, whose vertices are (reorientation classes of) UOMs of rank~$r$ and of order $n$. Two reorientation classes are connected if and only if there exists a mutation between them.
The graph $\mathcal{G}^{n,r}$ is motivated by the topological representation of OMs. In particular, by Ringel's Homotopy Theorem~\cite{Rin-56,Rin-57} it follows that $\mathcal{G}^{n,3}$ is connected. Moreover, the subgraph of $\mathcal{G}^{n,r}$ induced by the realizable UOMs is connected by~\cite{Rou-88}.
Las Vergnas' conjecture would imply that $\mathcal{G}^{n,r}$ has minimum degree at least $1$ (where loops can occur). A much stronger affirmation appears in~\cite{Rou-88}:

\begin{conj}[Cordovil-Las Vergnas]\label{conj:cordovil}
For all $r,n$ the graph $\mathcal{G}^{n,r}$ is connected.
\end{conj}

In Section~\ref{sec:mut} we introduce two variations of $\mathcal{G}^{n,r}$ that seem natural in the graph theoretic language of OMs, denoted $\overline{\mathcal{G}}^{n,r}$ and $\underline{\mathcal{G}}^{n,r}$. In the graph $\underline{\mathcal{G}}^{n,r}$ the vertices are graph isomorphism classes, instead of the reorientation classes, while in $\overline{\mathcal{G}}^{n,r}$ the vertices are all the OMs of rank $r$ and order $n$ (over the same ground set). Indeed, it is easy to see that the connectivity of $\overline{\mathcal{G}}^{n,r}$ implies the connectivity of $\mathcal{G}^{n,r}$ and the connectivity of $\mathcal{G}^{n,r}$ implies that $\underline{\mathcal{G}}^{n,r}$ is  connected. (Observation~\ref{obs:connectivities}).  

Our results with respect to Conjecture~\ref{conj:cordovil} include that $\overline{\mathcal{G}}^{n,3}$ is connected, which is a consequence of Ringel's Homotopy Theorem~\cite{Rin-56,Rin-57} and strengthens the fact that $\mathcal{G}^{n,3}$ is connected (Proposition~\ref{prop:pseudoline}). Moreover, we show that connectivity of $\underline{\mathcal{G}}^{n,r}$ implies connectivity of $\mathcal{G}^{n,r}$ (Proposition~\ref{prop:mut}). Since $\underline{\mathcal{G}}^{n,r}$ is much smaller that $\mathcal{G}^{n,r}$, the fact that Conjecture~\ref{conj:cordovil} is closed under duality, see~\cite[Exercise 7.9]{bjvestwhzi-93}, allows us to verify Conjecture~\ref{conj:cordovil} for all $n\leq 9$, computationally. 
See Table~\ref{tab:iso} for orders of the graphs $\underline{\mathcal{G}}^{n,r}$ and Figure~\ref{fig:mutations} for a depiction of $\underline{\mathcal{G}}^{8,4}$.

\begin{thmintro}
The Cordovil-Las Vergnas Conjecture holds for UOMs of order at most $9$.
%graph $\overline{\mathcal{G}}^{n,3}$ is connected for arbitrary $n$, as well as $\mathcal{G}^{n,r}$ for all $n\leq 9$.
\end{thmintro}

In the last part of the paper (Sections \ref{sec:sec5corners}, \ref{sec:antipodal}) we explore how the notion of simpliciality generalizes to classes of partial cubes containing OMs.
First we consider Complexes of Oriented Matroids (COMs)~\cite{Ban-18}. This class can be seen as a common generalization of OMs as well as,e.g.,~distributive lattices, antimatroids, median graphs, CAT(0) cube complexes, linear extensions of posets, pseudoline arrangements, hyperplane arrangements, Pasch graphs, and Lopsided Sets (LOPs). COMs have been acknowledged already several times in their short existence, see~\cite{margolis2015cell,itskov2018hyperplane,Bau-16,alex2020oriented,DKR21,padrol2021sweeps,R20,Hoch-19}. As their name suggest, COMs may be seen as complexes whose cells are OMs. In terms of the tope graph $G$ of a COM its OM cells can be identified as its so-called antipodal subgraphs. If a COM is a LOP, then these cells are just subgraphs isomorphic to hypercubes, see Section~\ref{sec:prel} for the details.

 We generalize the notion of simpliciality from OMs to COMs as follows: a vertex $v\in G$ is \emph{simplicial} if it is contained in a unique maximal antipodal subgraph $A\subseteq G$ and $\deg(v)=r(A)$. In LOPs simplicial vertices are usually called \emph{corners} and have important implications for results in computational learning theory, see~\cite{Cha-18}.
The notion of corners in COMs allows us to prove the existence of simplicial vertices in certain COMs.
%  However, wanting to keep these properties going back from corners in LOPs we get a notion of corners in COMs, that is stronger than having a simplicial vertex\comment{should we argue this somewhere}.
 In Section~\ref{sec:corners} we show that  realizable COMs (Proposition~\ref{prop:realizable}), COMs of rank~$2$ (Theorem~\ref{thm:peeling}), and hypercellular graphs (Theorem~\ref{thm:hypercellular}) admit \emph{corner peelings}, which is a way of deconstructing a COM in a well-behaved way. This result generalizes the corresponding results on realizable LOPs~\cite{Tra-04}, LOPs of rank~$2$~\cite{Cha-18}, and bipartite cellular graphs~\cite{Ban-96}. Furthermore, together with  the examples from~\cite{Ric-93,Bok-01,Tra-04} this yields locally realizable COMs that are not realizable and refutes a conjecture of~\cite[Conjecture 2]{Ban-18} (Remark~\ref{rem:notrealizable}). 

Lastly, we consider the question of generalizing  Conjecture~\ref{conj:lasvergnas} to antipodal partial cubes. First we consider the case where the rank is at most $3$. In this case the distinction between general antipodal partial cubes and OMs is easy, since OMs of rank at most $3$ are exactly the planar antipodal partial cubes~\cite{Fuk-93}. Based on this, for rank~$3$ OMs Conjecture~\ref{conj:lasvergnas} can be deduced from Euler's Formula and the fact that partial cubes are triangle-free. We extend this result by showing that all (hence also the non-planar) antipodal partial cubes of rank at most~3 have simplicial vertices in Section~\ref{sec:antipodal}.

Conjecture~\ref{conj:lasvergnas} has been verified for UOMs of rank~$4$ up to 12 elements~\cite{Bok-01}. In Section~\ref{sec:antipodal} we prove that even general antipodal partial cubes of isometric dimension up to $7$ have simplicial vertices. %This is one of the motivations why we wish to speak about (C)OMs as graphs (partial cubes), since \emph{it could be the case that the property of having a simplicial tope is purely graph theoretical}, i.e.~
It could be that every antipodal partial cube of rank $r$ has minimum degree at most $r$.

\begin{thmintro}
Realizable COMs, COMs of rank~$2$ and hypercellular graphs have corner peelings, thus in particular simplicial vertices. The minimum degree of antipodal partial cubes of rank~$3$ or of isometric dimension at most $7$ is at most their rank, i.e., they satisfy the Las-Vergnas property. 
\end{thmintro}

\section{Preliminaries}\label{sec:prel}

Based on a recent result from ~\cite{Kna-17}, we start by presenting OMs, Affine Oriented Matroids (AOMs) and Uniform Oriented Matroids (UOMs)  solely as graphs, in particular as a subfamily of partial cubes. A reader familiar with these concepts can avoid this view and read the paper nevertheless. % (first by skipping to Conjecture \ref{conj:lasvergnas} in this introduction).
We begin with necessary standard definitions from Metric Graph Theory:
%For a vertex $v$ of the hypercube we call the vertex with all its coordinates flipped the \emph{antipode} of $v$. %We denote the antipode of $v$ sometimes by $-v$. 

As stated above, a \emph{partial cube} is an isometric subgraph of a hypercube, where the minimal dimension it embeds into is called its \emph{isometric dimension}. The notion is well defined since an embedding of a partial cube is unique up to automorphisms of the hypercube~\cite[Chapter 5]{Ovc-11}.  As mentioned above, we call a partial cube $G$ of isometric dimension $n$ \emph{antipodal} if when embedded in $Q_n$ for every vertex $v$ of $G$ also its \emph{antipode} $-v$, i.e., the vertex with all coordinates flipped, is in $G$. 
A subgraph $H$ of $G$ is \emph{convex}, if all the shortest paths in $G$ connecting two vertices of $H$ are also in $H$. A convex subgraph $H$ of a partial cubes $G$ is a partial cube as well, since every convex (or isometric) subgraph of an isometric subgraph of a hypercube is isometric in the hypercube as well. In fact the embedding of $H$ is inherited from $G$. We say a subgraph $H$ of $G$ is \emph{antipodal}, if it is a convex subgraph of $G$ and an antipodal partial cube on its own (where antipodes in $H$ are considered with respect to its embedding in a minimal hypercube). Finally, we call a subgraph $H$ of a graph $G$ \emph{gated} if for every vertex $x$ of $G$ there exists a vertex $v_x$ of $H$ such that for every $u\in H$ there is a shortest path in $G$ from $x$ to $u$ passing through $v_x$. Gated subgraphs are convex \cite[Section 12.1]{ham-11}. %We will explain later on, why sometimes $v_x$ will be denoted $H\circ x$.
As mentioned earlier, embedding a partial cube $G$ in $Q_n$ yields a partition of the edges of $G$ into so-called \emph{$\Theta$-classes} corresponding to the dimensions of $Q_n$. Each $\Theta$-class consists of the edges corresponding to a flip of a fixed coordinate. We denote by $E_f$ a $\Theta$-class of $G$ where $f$ corresponds to a coordinate in $\{1, \ldots, n\}$. 
% is a representative edge of $E_f$ \comment{or $f$ is just an index as we did now in the COM paper}. 
We write $E_f^+$ for the vertices having the coordinate $f$ equal to $+$ and analogously define $E_f^-$. The sets $E_f^+$ and $E_f^-$ are called \emph{halfspaces} of $G$ and induce graphs that are partial cubes as well since they are already isometrically embedded in a hypercube. 
If $\{E_{f_1}^{s_1}, E_{f_2}^{s_2}, \ldots E_{f_k}^{s_k}\}$ are halfspaces in a partial cube $G$, for $f_i$ coordinate and $s_i \in \{+,-\}$, $1\leq i \leq k$, the partial cube induced by the intersection $\cap_{i=1}^k E_{f_i}^{s_i}$ is sometimes called a \emph{restriction} of $G$. Note that in the standard language of oriented matroids, $\Theta$-classes correspond to \emph{elements} of a simple OM, hence also isometric dimension is sometimes just called  the number of elements. %The resulting graph is sometimes also called an elementary \emph{restriction} and is denoted as $\rho_{f^+}(G)$.

Starting from a graph theoretical perspective we state a characterization -- that serves as a definition for now -- of (simple)  OMs, AOMs, UOMs, LOPs, and COMs.

\begin{thm}[\cite{CKP22,daS-95,Kna-17}]\label{thm:char}
 There is a one to one correspondence between the classes of (simple):
 \begin{itemize}
  \item[(i)]   OMs and antipodal partial cubes whose antipodal subgraphs are gated~\cite{Kna-17,daS-95},
  \item[(ii)]  AOMs and halfspaces of antipodal partial cubes whose antipodal subgraphs are gated~\cite{Kna-17},
  \item[(iii)] UOMs and antipodal partial cubes whose proper antipodal subgraphs are  hypercubes~\cite{Kna-17,CKP22}.
\item[(iv)]  LOPs and partial cubes whose antipodal subgraphs are  hypercubes~\cite{Law-83},
 hypercubes~\cite{Kna-17,CKP22}.

\item[(v)]  COMs and partial cubes whose antipodal subgraphs are gated~\cite{Kna-17},
 
 \end{itemize}
\end{thm}

While the proof of the above theorem is rather complicated, the correspondence with the standard definition of the OMs, AOMs, UOMs through covectors $\mathcal{L} \subset \{+,-,0\}$ (see below) is simple. One direction we already described above, namely given $\mathcal{L}$ the vertices of the tope graph are the elements without $0$-entries (topes) inducing a subgraph in $Q_{E}$. Conversely, one can get the covectors from the tope graph by associating to its antipodal subgraphs sign-vectors, that encode their relative position to the halfspaces. This is explained in more detail below.

An operation that is well known in the study of partial cubes is a (partial cube) \emph{contraction} $\pi_f$ that for
a coordinate $f$ contracts all the edges in a $\Theta$-class $E_f$. The family of partial cubes is closed under the operation of contraction as well as are the families of OMs, AOMs (in their graph representation) and the class of antipodal partial cubes, see~\cite{Kna-17}.
%Unfortunately, the established terminology in partial cubes does not match the standard terminology of covectors.
The partial cube contraction operation corresponds to \emph{deletion} in covector systems. %Conversely, a contraction of a covector system coincides with so called zone graphs in partial cubes that we define later. When speaking about these operations we will point out which terminology we are using.

The \emph{rank} $r(G)$ of a partial cube $G$ is the largest $r$ such that $G$ can be transformed to $Q_r$ by a sequence of partial cube contractions. The definition of rank in oriented matroid theory is equivalent, see~\cite{daS-95,CKP22}. Furthermore, notice that viewing the vertices of $G\subseteq Q_n$ as a set $\mathcal{S}$ of subsets of $\{1,\ldots, n\}$, $r(G)$ coincides with the VC-dimension of $\mathcal{S}$, see~\cite{VC-15}. The latter correspondence has led to some recent interest in partial cubes of bounded rank, see~\cite{CKP20}. See Figure~\ref{fig:desargues} for an antipodal partial cube of rank~$3$.

\begin{figure}[ht]
\centering
\includegraphics[width=.6\textwidth]{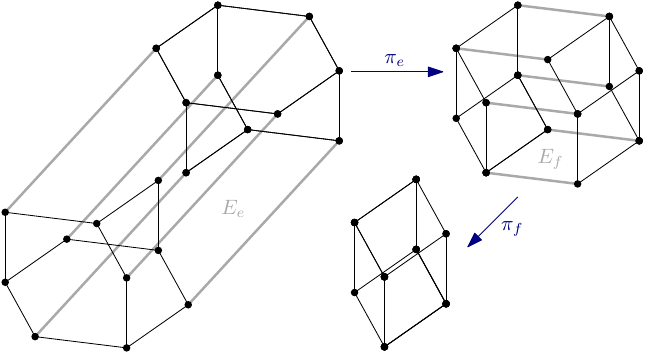}
\caption{Contracting an antipodal partial cube to $Q_3$.}
\label{fig:desargues}
\end{figure}

The inverse of a contraction is an {expansion}. In fact one can look at expansions in partial cubes the following way: if $H$ is a contraction of $G$, i.e.~$H=\pi_e(G)$, then one can consider in $H$ sets $H_1=\pi_e(E_e^+)$ and $H_2=\pi_e(E_e^-)$. Notice that they completely determine the expansion, since $G$ can be seen as a graph on the disjoint union of $H_1$ and $H_2$ where edges between them correspond to $H_1\cap H_2$. By \emph{expansion} of $H$ we refer to the subgraphs $H_1, H_2$ and sometimes to $G$. 
%In case that $H_1=H$ or $H_2=H$ we say that the expansion is \emph{peripheral}. A peripheral expansion is \emph{proper} if $H_1\neq H_2$.
 If $G$ and $H$ are OMs, then we say that $H_1,H_2$ is an \emph{OM-expansion}, also called \emph{single-element extension} in OM theory (we define OM-expansion in terms of covectors as well, see below).
%\comment{hard to grasp here: In fact if one defines $G$ and $H$ as sets of covectors, then $H_1,H_2 \subset H$ are subsets of covectors each being an AOM.}
More generally, if $H$ is an antipodal partial cube, then an expansion $H_1, H_2$ of $H$ is \emph{antipodal} if the expanded graph $G$ is again antipodal. It is well-known, see e.g~\cite{Kna-17,Pol-19}, that $G$ is antipodal if and only if $H_1=-H_2$. Clearly, a OM-expansion is antipodal. An OM-expansion $H_1, H_2$ is \emph{in general position} if each maximal proper antipodal subgraph is either  completely in $H_1$ or completely in $H_2$ but not in $H_1\cap H_2$. We reestablish the definition of the general position in terms of covectors below.
%\comment{hard to grasp here: In terms of OMs as covectors this is equivalent to $H_1$ and $H_2$ covering all the covectors.} 
Expansions in general position are central to the notion of Mandel OMs as well as to the definition of corners in COMs, i.e., in both Section~\ref{sec:mandel} and Section~\ref{sec:corners}. It is well-known, that every OM admits an expansion and indeed even an expansion in general position, see~\cite[Lemma 1.7]{San-02} or~\cite[Proposition 7.2.2]{bjvestwhzi-93}.

We introduced COMs and their subclasses as graphs. Nevertheless, for certain results it is convenient to define covectors and cocircuits of COMs which are standard ways to introduce OMs. 
%Apart from topes of a COM in this paper we will talk about the set of all covectors and in particular cocircuits.
 The covectors are represented as a subset $\covectors\subseteq \{+,-,0\}^n$ and have to satisfy certain axioms in order to encode a COM, OM, AOM, LOP or UOM. If $X \in \covectors$ and $e \in [n]$ is a coordinate of $X$, we shall write $X_e \in \{+, -, 0\}$ for the value of $X$ in coordinate $e$.
When considering tope graphs, one restricts usually to simple systems. Here, a system of sign-vectors $\covectors$ is \emph{simple} if it has no ``redundant'' elements, i.e., for each $e \in [n]$,  $\{X_e: X\in \covectors\}=\{+, -,0 \}$ and for each pair $e\neq f$ in $[n]$,  there exist $X,Y \in \covectors$ with $\{X_eX_f,Y_eY_f\}=\{+, -\}$. We will assume simplicity without explicit mention.  By the graph-theoretical representation of COMs given in~\cite{Kna-17}, the covectors correspond to the antipodal subgraphs of a COM $G$. Indeed, in a partial cube every convex subgraph is an intersection of halfspaces, see e.g.~\cite{Alb-16}, and one can assign to any convex subgraph $H$ a unique sign-vector $X(H)\in\{0,+,-
\}^n$ by setting for any coordinate $e\in\{1,\ldots, n\}$: $$X(H)_e=\begin{cases}
                                   + & \text{if }X\subseteq E_e^+,\\ 
                                   - & \text{if }X\subseteq E_e^-,\\
                                   0 & \text{otherwise.}\\
                                  \end{cases}$$

This correspondence yields a dictionary in which important concepts on both sides, graphs and sign-vectors translate to each other. We have spoken about isometric dimension and rank. Another instance of this is that if the $v\in G$ is a vertex of an antipodal partial cube, then for its antipode $u$ we have $X(u)=-X(v)$. Thus, we will often denote the antipodes of a set of vertices $H$ just by $-H$. Another noteworthy instance is the relation of gates and composition. The \emph{composition} of two sign vectors $X,Y\in\{0,+,-
\}^n$ is defined as the sign-vector obtained by setting for any coordinate $e\in\{1,\ldots, n\}$: $$(X\circ Y)_e=\begin{cases}                                   X_e & \text{if }X_e\neq 0,\\ 
Y_e & \text{otherwise.}\\                                  \end{cases}$$

\begin{prop}{\cite{Kna-17}}\label{prop:gates}                                 If $H$ is a gated subgraph of a partial cube $G$, then for a vertex $v$ with gate $u$ in $H$, we have $X(u)=X(H)\circ X(v)$.
\end{prop}

% Proposition~\ref{prop:gates} justifies, why for the gate of $v$ in $H$ we sometimes just write $H\circ v$. 
A set of covectors forms a COM if it satisfies the following two axioms:

\begin{itemize}
\item[{ (FS)}] $X\circ -Y \in  \covectors$  for all $X,Y \in  \covectors$.
\item[{ (SE)}]  for each pair $X,Y\in\covectors$ and for each $e\in [n]$ such that $X_eY_e=-1$,  there exists $Z \in  \covectors$ such that
$Z_e=0$  and  $Z_f=(X\circ Y)_f$  for all $f\in [n]$ with $X_fY_f\ne -1$.
\end{itemize}

It is easy to see, check e.g.~\cite{Ban-18}, that a set of covectors is an OM if and only if it is a COM satisfying additionally:

\begin{itemize}
\item[{ (Z)}] $\mathbf{0} \in  \covectors$.
\end{itemize}

For similar axiomatizations for LOPs, AOMs, and UOMs we refer to~\cite{Ban-18,Bau-16}. We now reestablish some definitions stated in terms of graphs in terms of covectors. What was defined as a partial cube contraction on the tope graph of an OM, corresponds to so called \emph{deletion} in OMs. In particular, if $\covectors$ is a set of covectors of an OM and $e\in[n]$ is an element of it, then $\covectors' = \covectors \backslash e = \{X \backslash e \mid X \in \covectors \}$ is a set of covectors of a smaller OM, that is said to the obtain by the deletion of $e$. On the contrary, of $\covectors$ is said to be a \emph{single-element extension} of $\covectors'$ (recall that in the graph setting it is named an expansion). The extension is said to be in \emph{general position} if the union of images of $\{X \in \covectors \mid X_e=+\}$ and $\{X \in \covectors \mid X_e=-\}$ under the deletion covers all the covectors of $\covectors'$. In fact, as we will see in Section \ref{sec:mandel}, the image of $\{X \in \covectors \mid X_e=+\}$ under deletion is maximal with the property that it induces a COM in $\covectors'$. We will use this to define a corner of an OM. For $G_1,G_2$ being the images of $\{X \in \covectors \mid X_e=+\}$ and $\{X \in \covectors \mid X_e=-\}$ under deletion, we will sometimes also say that the expansion is along $G_1$ and $G_2$.

We will content ourselves with the definition of OMs and COMs with covectors and for other classes use Theorem~\ref{thm:char} as a definition. In order to formulate (SE) in terms of antipodal subgraphs, we give some more terminology about how subgraphs and $\Theta$-classes can relate.
We will say that a $\Theta$-class $E_f$ \emph{crosses} a subgraph $H$ of $G$ if at least one of the edges in $H$ is in $E_f$. Moreover, $E_f$ \emph{separates} subgraphs $H,H'$ if $H\subseteq E_f^+$ and $H'\subseteq E_f^-$ or the other way around. We collect the coordinates separating $H$ and $H'$ in the set $S(H,H')$ --  usually called the \emph{separator}. Let us now state the axiom of \emph{strong elimination for graphs}:

\begin{itemize}
 \item[(SE)] Let $X,Y$ be antipodal subgraphs of $G$ and $E_e$ a $\Theta$-class such that $X\subseteq E_e^+$ and $Y\subseteq E_e^-$, i.e., $e\in S(X,Y)$. There is an antipodal subgraph $Z$ of $G$ that is crossed by $E_e$ and for all $f\neq e$ we have:
 
 \vspace{-.3cm}
 
 \begin{itemize}
  \item if $X, Y\subseteq E_f^{+}$ then $Z\subseteq E_f^{+}$, and if $X, Y\subseteq E_f^{-}$ then $Z\subseteq E_f^{-}$
  \item if $E_f$ crosses one of $X, Y$ and the other is a subset of $E_f^{+}$ then $Z\subseteq E_f^{+}$, and if $E_f$ crosses one of  $X, Y$ the other is a subset of $ E_f^{-}$ then $Z\subseteq E_f^{-}$
  \item if $E_f$ crosses both $X$ and $Y$, then it crosses $Z$.
 \end{itemize}
\end{itemize}

In some parts of the paper we will abuse a bit the distinction between covectors and antipodal subgraphs in the way that we have suggested in the above definitions.

\begin{figure}[ht]
\centering
\includegraphics[width=\textwidth]{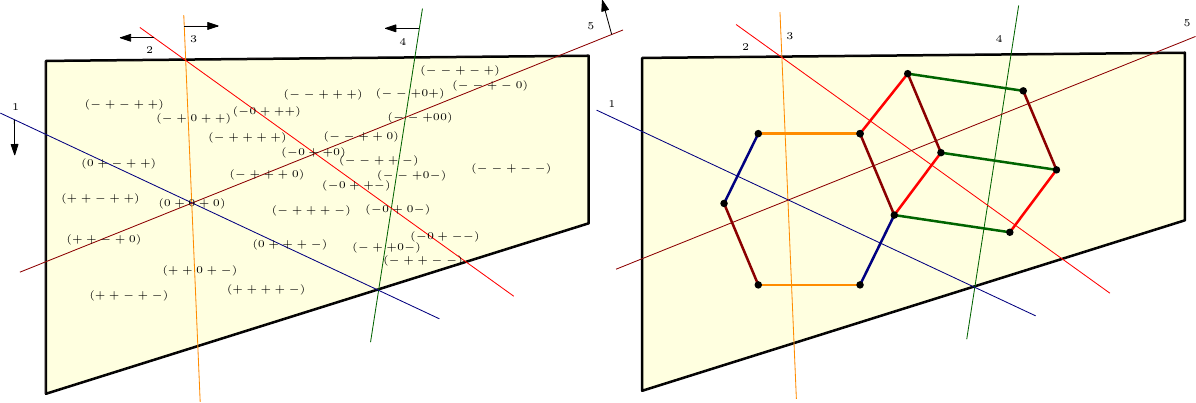}
\caption{A realizable COM and its tope graph. Its bounded faces, edges and vertices are the antipodal subgraphs.}
\label{fig:exampletopegraph}
\end{figure}
 
Some COMs, OMs, AOMs, and LOPs can be particularly nicely represented by a geometric construction. Let $\{H_1, \ldots, H_n\}$ be a set of hyperplanes in an Euclidean space $\mathbb{R}^d$ and $C$ a full-dimensional open convex set in $\mathbb{R}^d$. The hyperplanes cut $C$ in $d$-dimensional chambers and to every point in $C$ one can associate a vector in $\{+,-,0\}^n$ that denotes its relative position to the hyperplanes (if they are given with positive and negative side). These are the covectors. One can form a graph $G$ whose vertices are the chambers where two chambers are adjacent if and only if they are separated by exactly one hyperplane. Such graphs are always COMs. If  $C= \mathbb{R}^d$ then the graph is an AOM, and if moreover all $\{H_1, \ldots, H_n\}$ cross the origin point, then the graph is an OM. If $\{H_1, \ldots, H_n\}$ are the coordinate hyperplanes of $\mathbb{R}^d$ and $C$ is arbitrary, then the graph is a LOP. If a COM, OM, AOM, or LOP $G$ is isomorphic to a graph obtained in this way, we call it \emph{realizable}. See Figure~\ref{fig:exampletopegraph} for a realizable COM and its tope graph. Realizable COMs embody many nice classes such as linear extension graphs of posets, see~\cite{Ban-18} and are equivalent to special convex neural codes, namely so-called stable hyperplane codes, see~\cite{itskov2018hyperplane,alex2020oriented}.

If $G$ is an OM, then the set of antipodal subgraphs is usually ordered by reverse inclusion to yield the \emph{big face lattice} $\mathcal{F}(G)$ whose minimum is $G$ itself.
This ordering in terms of covectors can be equivalently defined by setting $X\leq Y$ for covectors $X,Y$, as the product ordering of $\{+,-,0\}^n$ relative to the ordering $0 \leq +, 0 \leq -$.
If $G$ is of rank~$r$, then it follows from basic OM theory that the $\mathcal{F}(G)$ is atomistic and graded, where $r + 1$ is the length of any maximal chain. Hence the latter can also be used to define the rank of $G$. The atoms, i.e. the covectors corresponding to the antipodal subgraphs of rank~$r-1$, are called \emph{cocircuits} in the standard theory of OMs.
%Moreover, an intersection of any two antipodal subgraphs is a (possibly empty) antipodal subgraph. We shall denote $X \prec Y$ if there exist no $Z$ with $X < Z < Y$. The antipodal subgraphs of rank~$r-1$ are themselves OMs and correspond to what is called  \emph{cocircuits} in the standard theory of OMs. In other words the \emph{cociruits} of $G$ are the atoms of $\mathcal{F}(G)$, i.e., the maximal proper antipodal subgraphs of $G$. 

Another graph associated to an OM $G$ of rank~$r$ is its \emph{cocircuit graph} of $G$, i.e., the graph $G^*$ whose vertices are the antipodal subgraphs of $G$ of rank~$r-1$ (equivalently cocircuits) where two vertices are adjacent if their intersection in $G$ is an antipodal subgraph of rank~$r-2$ (i.e. two cocircuits $X,Y$ are adjacent if there exists a covector $Z$ with $X,Y \prec Z$). We denote the cocircuit graph of $G$ by $G^*$, since it generalizes planar duality in rank~$3$, see~\cite{Fuk-93}, however in higher rank~$(G^*)^*$ is not well-defined, because the cocircuit graph does not uniquely determine the tope graph, see~\cite{cordovil2000cocircuit}. There has been extensive research on cocircuit graphs~\cite{bab-01,omax,omcubic,mon-06}. However their characterization and recognition remains open. Cocircuit graphs play a crucial role for the notion of Euclideaness and Mandel in Section~\ref{sec:mandel}.

In a general COM $G$, the poset $\mathcal{F}(G)$ remains an upper semilattice, since antipodal subgraphs are closed under intersection, but there is no minimal element. There are different possible notions of cocircuits that allow to axiomatize COMs, see~\cite{Ban-18}. We consider cocircuits in the setting of \emph{pure} COMs $G$, i.e., all maximal antipodal subgraphs of $G$ are of the same rank and $G^*$ is connected. If $G$ is a non-antipodal pure COM, then its cocircuits are just the maximal antipodal subgraphs. We will introduce the cocircuit graph of pure COMs in Section~\ref{sec:corners}, where it will serve for proving the existence of corners in COMs of rank~$2$.

\section{Mutation graphs of uniform oriented matroids}\label{sec:mut}

Let us define mutations of UOMs through their tope graph. If $v$ is a simplicial vertex in a UOM $G$ of rank~$r$, then $v$ is contained in a unique convex hypercube minus a vertex, let us denote it by $Q^-_{r}$. This is a well known assertion and it directly follows from Lemma~\ref{lem:simplicialvertex}. If one fills in the missing vertex of $Q^-_{r}$ and instead removes $v$ and does the same to the antipodes of the vertices in $G$, one obtains a new UOM $G'$ of rank~$r$. This operation is called a \emph{mutation}. Hence, a mutation is an operation that transforms a UOM into another UOM. A simple analysis shows that the operation is reversible, i.e.~the inverse operation is also a mutation, and that the rank of both UOMs is equal.

Thus, one can now consider a \emph{mutation graph} whose vertices are UOMs embedded into $Q_n$, for some $n\in \mathbb{N}$  of fixed rank~$r$ and edges are corresponding to mutations. In fact, one can consider three mutation graphs corresponding to the different notions of equivalence of OMs, that will now be introduced.

As stated above, the isometric embedding of a partial cube into a hypercube of minimum dimension is unique up to automorphisms of the hypercube, see e.g.~\cite[Chapter 5]{Ovc-11}. Indeed, partial cubes $G_1,G_2\subseteq Q_n$ are isomorphic as graphs if and only if there is $f\in\mathrm{Aut}(Q_n)$ such that $f(G_2)=G_1$. This leads to the fact that isomorphisms of simple OMs and their tope graphs correspond to each other. Since (unlabeled, non-embedded) isomorphic tope graphs are sometimes considered as equal, this allows to speak about isomorphism classes of OMs in a natural way. The situation is the same if one consider OMs in the standard definition since an isomorphism of an OM directly translates to an isomorphism of its tope graph, see e.g.~\cite{bjvestwhzi-93,Bjo-90}.

A more refined notion of isomorphisms for partial cubes are reorientations. For partial cubes $G_1,G_2\subseteq Q_n$ one says that $G_1$ is a \emph{reorientation} of $G_2$ if there is $f\in\mathbb{Z}_2^n\subseteq\mathrm{Aut}(Q_n)$ such that $f(G_2)=G_1$, i.e., $f$ only switches signs. This yields an equivalence relation whose classes are called \emph{reorientation classes}.

We give the three resulting graphs:
\begin{itemize}
\item $\overline{\mathcal{G}}^{n,r}$ is the graph whose vertices are UOMs of rank~$r$ and isometric dimension $n$, embedded into $Q_n$. Two graphs are connected if and only if there exists a mutation between them.
\item $\mathcal{G}^{n,r}$ is the graph whose vertices are reorientation classes of UOMs of rank~$r$ and isometric dimension $n$ embedded into $Q_n$. Two reorientation classes are connected if and only if there exists a mutation between them.
\item $\underline{\mathcal{G}}^{n,r}$ is the graph whose vertices are graph isomorphism classes of UOMs of rank~$r$ and isometric dimension $n$. Two classes are connected if and only if there exists a mutation between them.
\end{itemize}

The three different mutation graphs are related as follows:

\begin{obs}\label{obs:connectivities}
 Let $0\leq r\leq n$. We have $\overline{\mathcal{G}}^{n,r}$ connected $\implies \mathcal{G}^{n,r}$ connected $\implies \underline{\mathcal{G}}^{n,r}$ connected. 
\end{obs}
\begin{proof}
 If $\overline{\mathcal{G}}^{n,r}$ is connected, then also $\mathcal{G}^{n,r}$ is, since there is a weak homomorphism from the first to the second, mapping an OM to its equivalence class. Similarly, if $\mathcal{G}^{n,r}$ is connected then also  $\underline{\mathcal{G}}^{n,r}$ is, since a reorientation can be seen as an isomorphism. 
\end{proof}

We start by analyzing the connectivity of mutation graphs for small rank. Since OMs of rank 1 or 2 are simply isomorphic to an edge or an even cycle, respectively, the first interesting case is when the rank of OMs is 3.
Our first result on the topic, together with Observation~\ref{obs:connectivities} settles the connectivity of all three graphs in rank $3$: 

\begin{prop}\label{prop:pseudoline}
For every $n$ the graph $\overline{\mathcal{G}}^{n,3}$ is connected.
\end{prop}
\begin{proof}
 It is a well known fact, implied by the Topological Representation Theorem~\cite{fo-la-78},  that OMs of rank~3 can be represented as pseudo-circle arrangements on a $2$-dimensional sphere. Further, this arrangement can be assumed to be centrally symmetric and then by restricting the sphere along an equator that does not contain any intersection of the pseudo-circles, one obtains a wiring diagram, see Figure~\ref{fig:pseudoline}, i.e., a set of pseudo-lines traversing from left to right pairwise intersecting exactly once. Note that different choices of equator yield different wiring diagrams, but the all represent the OM in the same way, and only differ in what line is the topmost entering on the left. 
 Then, every pseudoline splits the plane into positive and negative half, and every point in the plane yields a covector, by associating to each pseudo-line a coordinate and putting a $+,-,0$ in it, depending on the point being on positive, negative side or on the pseudo-line. This, way for each covector $Y$ of the OM exactly one of $Y,-Y$ is represented. Maximal cells correspond to topes where two are adjacent if the cells intersect on a line-segment. The OM being uniform corresponds to arrangement being \emph{simple}, i.e., no three lines cross in a point.
 A \emph{mutation} of the wiring-diagram means pulling a bounding line of a simplicial cell over its opposite corner. A special case is if the simplicial cell is crossed by the bounding vertical lines, see Figure~\ref{fig:pseudoline}. This operation corresponds to what we described as mutation in tope graphs.

%  By  any two simple pseudo-line arrangements on a sphere can be transformed one into another by performing mutations. Since it does not deal with orientations, Ringel's Homotopy Theorem  implies that $\mathcal{G}^{n,3}$ is connected. As stated in Observation~\ref{obs:connectivities}, this implies that also $\underline{\mathcal{G}}^{n,3}$ is connected.
% 

 We now use the proof of Ringel's Homotopy Theorem~\cite{Rin-56,Rin-57} as shown in~\cite[Section 6.4]{bjvestwhzi-93}. Ringel shows that any two labeled simple wiring diagrams can be transformed into another by performing mutations.
 
 Suppose now that we have two UOMs of rank three $\mathcal{M}$ and $\mathcal{M}'$ and want to transform one into the other by mutations. Both can be represented by wiring diagrams, as stated above. Then by Ringel's Theorem we can assume that $\mathcal{M}$ can be transformed into a reorientation of $\mathcal{M}'$. So to prove that $\overline{\mathcal{G}}^{n,3}$ is connected we only need to consider the case where $\mathcal{M}$ and $\mathcal{M}'$ differ in the reorientation of one element $e$. Represent $\mathcal{M}$ as a wiring-diagram $\mathcal{A}$, where $e$ is the top-element on the left and construct a wiring diagram $\mathcal{A}'$ as shown in Figure~\ref{fig:pseudoline}, with $e$ being the bottom element on the left. Note that $\mathcal{A}'$ represents $\mathcal{M}'$, and the mutations transforming $\mathcal{A}$ into $\mathcal{A}'$ by Ringel's Theorem push the line $e$ without changing its orientation. Thus, we have obtained $\mathcal{M}'$ from $\mathcal{M}$ by mutations.
\end{proof}

% By Observation~\ref{obs:connectivities}, this implies that also $\mathcal{G}^{n,3}$ and $\underline{\mathcal{G}}^{n,3}$ are connected.

\begin{figure}[ht]
\centering
\includegraphics[width=\textwidth]{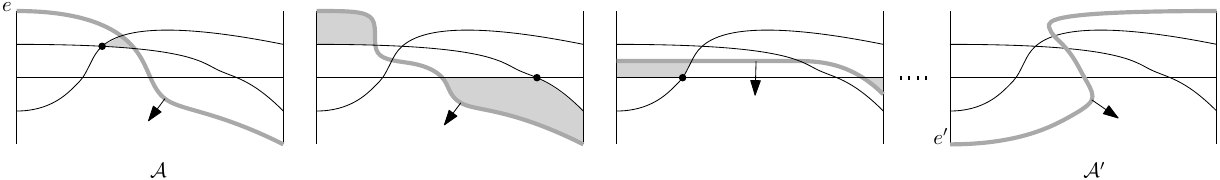}
\caption{How to construct the new pseudoline arrangement in the proof of Proposition~\ref{prop:pseudoline}. The arrow points to the positive side of $e$. The grey cell is the simplicial cell on which a mutation is applied.}
\label{fig:pseudoline}
\end{figure}

We show that two of the open questions about the connectivity of mutation graphs are equivalent.

\begin{table}[ht]
\center
\begin{tabular}{l||c|c|c|c|c|c|c|c|c}
$n\backslash r$  & 2 & 3 & 4 & 5 & 6 & 7 & 8 & 9 & 10  \\ \hline \hline
2 & 1 & 1 & 1 & 1 & 1 & 1 & 1 & 1 & 1 \\ \hline
3 &  & 1 & 1 & 1 & 4 & 11 & 135 & 482 & 312356 \\\hline
4 &  &  & 1 & 1 & 1 & 11 & 2628 & 9276595 & ? \\\hline
5 &  &  &  & 1 & 1 & 1 & 135 & 9276595 & ? \\\hline
6 &  &  &  &  & 1 & 1 & 1 & 4382 & ? \\\hline
7 &  &  &  &  &  & 1 & 1 & 1 & 312356 \\\hline
8 &  &  &  &  &  &  & 1 & 1 & 1 \\\hline
9 &  &  &  &  &  &  &  & 1 & 1 \\\hline
10 &  &  &  &  &  &  &  &  & 1 
\end{tabular}
\caption {Known orders of $\underline{\mathcal{G}}^{n,r}$, retrieved from \url{http://www.om.math.ethz.ch/}.}
\label{tab:iso}
\end{table} 

\begin{prop}\label{prop:mut}
If $\underline{\mathcal{G}}^{n,r}$ is connected, then $\mathcal{G}^{n,r}$ is connected.
\end{prop}
\begin{proof}
% As stated in Observation~\ref{obs:connectivities}, the mapping from $\mathcal{G}^{n,r}$ to $\underline{\mathcal{G}}^{n,r}$ defined by mapping the class of OMs up to the reorientation into their isomorphism classes is a weak homomorphism of graphs. Hence if $\mathcal{G}^{n,r}$ is connected, so is $\underline{\mathcal{G}}^{n,r}$.

Notice that the property of being a realizable OM is independent of reorientation or permuting the elements. If $\underline{\mathcal{G}}^{n,r}$ is connected, then there exists a sequence of mutations from any $[A]\in \underline{\mathcal{G}}^{n,r}$ to a realizable class $[B]\in\underline{\mathcal{G}}^{n,r}$. This sequence can then be lifted to a sequence of mutations from any $A'\in \mathcal{G}^{n,r}$ to a realizable $B'\in\mathcal{G}^{n,r}$. This proves that there exists a path from every $A'\in \mathcal{G}^{n,r}$ to a reorientation classes of realizable OMs. Since by~\cite{Rou-88} the induced subgraph of all realizable classes  in $\mathcal{G}^{n,r}$ is connected, this proves that $\mathcal{G}^{n,r}$ is connected.
\end{proof}

Proposition~\ref{prop:mut} allows to approach Conjecture~\ref{conj:cordovil} for small values of $n$ and $r$ from the computational perspective, since it allows computations on the smaller graph $\underline{\mathcal{G}}^{n,r}$. To provide an idea of the computational weight of this task, Table~\ref{tab:iso} shows the known orders of such graphs $\underline{\mathcal{G}}^{n,r}$ for small $n$ and $r$. Moreover, Figure~\ref{fig:mutations} displays the graph $\underline{\mathcal{G}}^{8,4}$.

\begin{figure}[ht]
\center
\includegraphics[width=.6\textwidth]{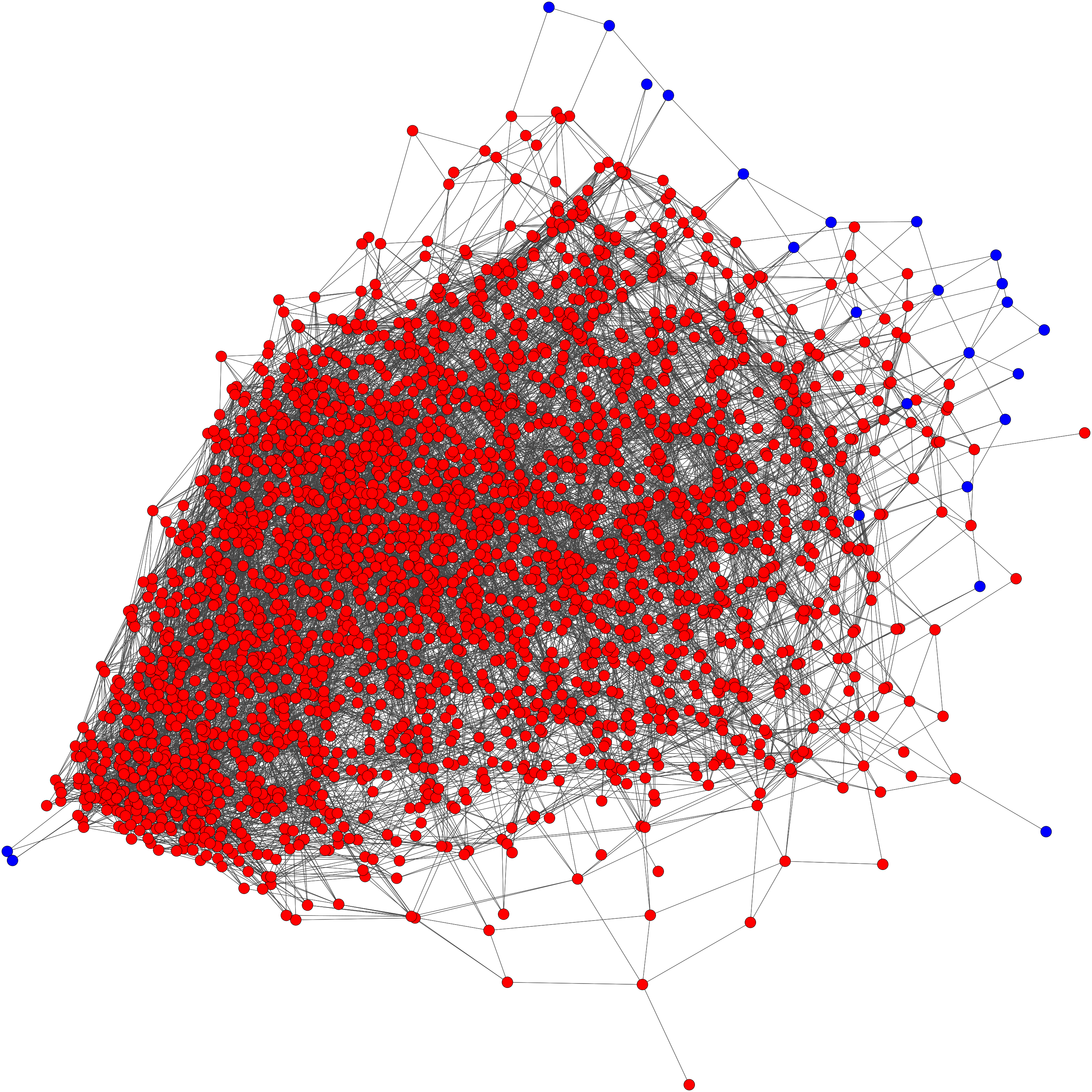}
\caption{The graph $\underline{\mathcal{G}}^{8,4}$ where red vertices are isomorphism classes of realizable UOMs and blue ones are non-realizable}
\label{fig:mutations}
\end{figure}

We verified computationally that for all the parameters from Table \ref{tab:iso} where the isomorphism classes of OMs are known, their mutation graph $\underline{\mathcal{G}}^{n,r}$ is connected. By Proposition \ref{prop:mut} also the corresponding $\mathcal{G}^{n,r}$ are connected. This was possible by considering UOMs as (tope)graphs in which finding possible mutations is easy, since only degrees of vertices need to be checked. We calculated the isomorphism class of the mutated graphs using the software Bliss \cite{junttila2007engineering} that is designed for calculations of isomorphisms of graphs. The computationally most demanding task was the graph $\underline{\mathcal{G}}^{9,4}$ where efficient graph representation was needed. 

\begin{cor}\label{cor:94}
The graph $\mathcal{G}^{n,r}$ is connected for $n\leq 9$.
\end{cor}

Testing graph isomorphism instead of OM-isomorphism was an essential ingredient in order to obtain Corollary~\ref{cor:94}. Checking connectivity of $\overline{\mathcal{G}}^{n,r}$ is far more demanding. We do not know anything about the connectivity of this graph beyond rank $3$, see Proposition~\ref{prop:pseudoline}. 

\section{Mandel's OMs and Euclideanness}\label{sec:mandel}

In this section we focus on Las Vergnas' and Mandel's Conjectures. The main result is that Mandel OMs are $\Theta$-Las Vergnas and therefore not all OMs are Mandel. In order to make the result accessible to readers familiar with the OM terminology as well as to readers more familiar with graphs we try to argue in the proof along both views. 

The concept of Euclideanness is based on the orientations of cocircuit graph introduced in Section~\ref{sec:prel}. Furthermore, the property of being Mandel relies on the definition of extensions in general position and Euclideanness.

Let $G$ be an OM of rank~$r$. Recall the definition of cocircuit graph $G^*$ from Section~\ref{sec:prel}.  Graph $G^*$ has cocircuits of $G$ for vertices, i.e. in graph theoretical language the antipodal subgraphs of $G$ of rank~$r-1$.  Two cocircuits $X,Y$ are adjacent if there exists a covector $Z$ with $X,Y \prec Z$ (hence $Z$ has rank $r-2$) in the big face lattice $\mathcal{F}(G)$, or equivalently if the intersection of the antipodal subgraphs corresponding to $X$ and $Y$ in $G$ is an antipodal subgraph of rank~$r-2$.

Consider now  a maximal subset $F$ of the $\Theta$-classes of $G$ such that there exist a covector $Z$ of rank $r-2$ with $Z_e=0$, for all $e\in F$. In other words, there exists an antipodal subgraph of rank $r-2$ crossed by the $\Theta$-classes in $F$. Let $X_1,\ldots,X_n$ be a maximal path in  $G^*$ such that $(X_i)_e=0$ for all $e\in F$, i.e. all the corresponding antipodal subgraph are crossed by the $\Theta$-classes in $F$. It follows from the topological representation of OMs~\cite{Man-82}, that $X_1,\ldots,X_n$ induce a cycle in $G^*$. Moreover, $X_{n/2+i}$ is the antipode of $X_{i}$ in $G$, for $1\leq i \leq n/2$. Furthermore, for each $\Theta$-class $E_f\notin F$ there are exactly two antipodal cocircuits $X_{i}$ and $X_{n/2+i}$ such that $(X_i)_f = (X_{n/2+i})_f=0$, meaning that $E_f$ crosses exactly two pairwise antipodal subgraphs corresponding to $X_{i}$ and $X_{n/2+i}$ on the cycle. The cocircuit graph $G^*$ is the (edge-disjoint) union of such cycles.

Let now $H$ be a halfspace  of $G$, i.e. $H$ is (the tope graph of) an AOM. The induced subsequence $X_k,\ldots,X_{\ell}$ of the above cycle is called a \emph{line} $L$ in $H$. This name comes from the fact that in the topological representation the sequence corresponds to a pseudo-line, see~\cite{Man-82}.
Let now $e$ be an element (or $E_e$ be a $\Theta$-class) of $H$. We say that $e$ (or $E_e$) \emph{crosses} a line $L$ of $H$, if there exists a unique $X_i$ on $L$ with $(X_i)_e = 0$, i.e. the corresponding antipodal subgraph  is crossed by $E_e$. Note that for a line $L$ of an AOM, if there exists $X_i$ with $(X_i)_e = 0$, then $X_i$ is either unique, or all cocircuits on $L$ are crossed by $E_e$. Equivalently, exactly one or all of the antipodal subgraphs on the line are crossed by $E_e$. This allows to define 
\emph{the orientation of line $L$ with respect to $e$ (or $E_e$)}: If $L$ is not crossed by $e$ (equivalently $E_e$) we leave its edges undirected. Otherwise, let $X_i$ be the element of $L$ that is crossed by it and assume that $X_j\subset E_e^-$ for $j<i$ and $X_j\subset E_e^+$ for $j>i$. We orient all edges of the form $X_j, X_{j+1}$ from $X_j$ to $X_{j+1}$. This is, the path $L$ is directed from $E_i^-$ to $E_i^+$.

The edges of the cocircuit graph $H^*$ of an AOM $H$ are partitioned into lines and for every element $e$ (or $\Theta$-class $E_e$) we obtain a partial orientation of $H^*$ by orienting every line with respect to $E_e$. Let us call this mixed graph the orientation of $H^*$ with respect to $E_e$.
Following Mandel~\cite[Theorem 6]{Man-82}, an AOM is \emph{Euclidean} if for every element $e$ (or $\Theta$-class $E_e$) the orientation of the cocircuit graph $H^*$ with respect to $e$ is \emph{strictly acyclic}, i.e., {any directed cycle (following undirected edges or directed edges in the respective orientation) consists of only undirected edges.} In other words, any cycle  that contains a directed edge contains one into each direction. %Euclidean AOMs are important since they allow a generalization of linear programming from realizable AOMs.

Following Fukuda~\cite{Fuk-82}, an OM is called \emph{Euclidean} if all of its halfspaces are Euclidean AOMs. Since non-Euclidean AOMs exist, see~\cite{Fuk-82,Man-82}, also non-Euclidean OMs exist. However, there is a larger class of OMs that inherits useful properties of Euclidean AOMs and that was introduced by Mandel \cite{Man-82}.
% If $G_1$ is an AOM corresponding to a half of an expansion in general position of an OM, then also in $G$ each line is crossed by all the $\Theta$-classes besides the ones that cross all $A_i \cap A_{i+1}$.\comment{dont get this sentence nor its usefulness} 
%To define this class we need some further definitions. 

We call an OM \emph{Mandel} if it has an expansion in general position such that $G_1$ (and implicitly $G_2$, since they are isomorphic with the antipodal map) are Euclidean AOMs. In terms of the covectors description, this can be stated as follows: if $\covectors$ is a set of covectors of an OM, and there exists $\covectors'$ a single-element extension in general position of it with respect to an element $e$, i.e.~$\covectors'\backslash e = \covectors$, with the property that $\{X \in \covectors' \mid X=+\}$ induces an Euclidean AOM, then $\covectors$ is said to be Mandel.
 Mandel~\cite[Theorem 7]{Man-82} proved (and it is up to today the largest class known to have this property) that these OMs satisfy the conjecture of Las Vergnas:

\begin{thm}[\cite{Man-82}]\label{thm:mandel}
 If an OM $G$ is Mandel, then it has a simplicial vertex.
\end{thm}

%As stated in the introduction, Mandel~\cite[Conjecture 8]{Man-82} conjectured that every OM is Mandel as a ``wishful thinking statement'', since with Theorem~\ref{thm:mandel} it would imply the conjecture of Las Vergnas (Conjecture~\ref{conj:lasvergnas}).
In the following we go along the lines of Mandel's proof of Theorem~\ref{thm:mandel} from \cite{Man-82} while refining it to such an extent, that we can disprove Conjecture \ref{conj:mandel}. In particular, we use a stricter version of the induction in the proof to obtain $\Theta$-Las Vergnas properties of Mandel OMs.

We shall repeatedly use the following fact, that is trivial with the covectors definition of OMs but is a bit more complicated to prove in the graph view. Nevertheless, it was also proved with graphs in \cite[Lemma 6.2]{Kna-17}.

\begin{obs}\label{lem:exp_antipodal}
Let $G, G'$ be OMs such that $G'$ is an OM deletion (i.e.~a partial cube contraction) of $G$, with respect to an element $e$ (a $\Theta$-class $E_e$). Then for every cocircuit $X'$ of $G'$ there exists a cocircuit $X$ in $G$, such that $X_f=X'_f$ for all $f\neq e$. In graph theoretical view, there exists a maximal antipodal subgraph of $G$ that contracts to the maximal antipodal subgraph corresponding to $X'$. The same holds for COMs.
\end{obs}

The following is a characterization of simplicial vertices in an OM, that can be found in~\cite[Proposition 5, page 320]{Man-82} and~\cite[Proposition 1.4]{lv-80}. We provide a slightly different formulation together with a formulation in terms of the tope graph. Recall that in an OM $G$ a vertex $v$ is simplicial if the degree of $v$ coincides with the rank of $G$. Also recall, that by definition $G$ is $\Theta$-Las Vergnas if for every $\Theta$-class $E_e$ of $G$, there exists a simplicial vertex incident with $E_e$. In terms of the covectors description, this means that for every element $e$ there exists a simplicial tope $X$ and a covector $Y$ with $Y\prec X$ such that $Y_e = 0$. We will say that $X$ is \emph{incident} with $e$.

\begin{lem}\label{lem:simplicialvertex}
Let $G$ be an OM of rank~$r$. A tope $v\in G$ is simplicial if and only if there is a set of elements $F$ of size $r$ and $r$ cocircuits, such that for each $e \in F$ there is a corresponding cocircuit $X$ with $X_f=0$ for $f \in F\setminus\{e\}$, $X_e \neq 0$ and $X \leq v$. In the graph interpretation, this is equivalent to $r$ $\Theta$-classes $F$ and $r$ antipodal subgraphs incident with $v$ such that each is crossed by $F\setminus\{E_e\}$ but not by $E_e$, each for a different $E_e$. In this case, $v$ is incident with elements ($\Theta$-classes) in $F$.
%
%maximal proper antipodal subgraphs incident with $v$ such that for each $E_e \in F$ there is a corresponding antipodal subgraph which is crossed by $F\setminus\{E_e\}$ but not by $E_e$. Moreover, in this case $v$ is incident with exactly the $\Theta$-classes of $F$ and exactly the aforementioned $r$ antipodal subgraphs.
\end{lem}

As a last basic ingredient for the Theorem \ref{thm:simplical_on_theta} we need the following. Let $G$ be the tope graph of a simple OM and $e$ one of its elements. Let $v$ be a tope in the positive halfspace of $G$ with respect to $e$. Since the lattice of covectors is atomistic and graded, there is a set of cocircuits such that their join is exactly $v$. In particular, for at least one of the latter cocircuits, say $X$, it must hold $X_e \neq 0$. 
%
%maximal proper antipodal subgraphs must be in $E_e^+$.
This is:
\begin{obs}\label{obs:coloop}
Let $e$ be an element ($E_e$ be a $\Theta$-class) of a simple OM $G$. There exists a cocircuit $X$ (a maximal proper antipodal subgraph $A$) such that $X_e \neq 0$ ($A$ is not crossed by $E_e$).
\end{obs}

% Let $E_e$ be a $\Theta$-class in an rank~$r$ OM $G$. Iteratively contract every $\Theta$ class in $G$ that does not correspond to a factor $K_2$, i.e. the graph is not isomorphic to the Cartesian product $K_2 \square G'$ with $E_e$ corresponding to $K_2$ edges. This procedure does not lower the rank \TODO{(missing argument why is that)} of the graph and the final obtained graph is isomorphic to a hypercube $Q_r$. In this $Q_r$ rank~$r-1$ antipodal subgraphs are just $r-1$ dimensional subcubes. Hence notice that $E_e$ crosses all but two rank~$r-1$ antipodal subgraphs in $Q_r$. Using Lemma \ref{lem:exp_antipodal} we deduce that there exist in $G$ at least two maximal antipodal subgraphs not crossed by $E_e$. Since $E_e$ was arbitrary, the latter holds for every $\Theta$-class. \TODO{I guess this point is also simpler from cocircuit axioms?}.
% 
% 

We are prepared to give the main theorem of the section.
In the proof we will make use of the \emph{direct sum} $G_1 \oplus G_2$ of OMs $G_1$ and $G_2$. In the OM language this is defined as the OM with covectors $\{(X_1 , X_2) \mid X_1 \in G_1, X_2 \in G_2\}$. Considering OMs as graphs, this is equivalent to the \emph{Cartesian product} of graphs $G_1, G_2$, being defined as the graph $G_1\B G_2$ whose vertices are $V(G_1) \times V(G_2)$ where two vertices $(x_1, y_1),(x_2, y_2)$ adjacent if and only if $x_1=x_2$ and $y_1$ is adjacent to $y_2$ in $G_2$, or $y_1=y_2$ and $x_1$ is adjacent to $x_2$ in $G_1$. We will denote by $K_2$ the complete graph on two vertices, or equivalently the unique simple OM with ground set of size $1$. 

\begin{thm}\label{thm:simplical_on_theta}
Let $G$ be a simple, Mandel OM of rank~$r$ and $e$ one of its elements ($E_e$ a $\Theta$-class of $G$). Then there is a simplicial tope (a vertex of degree $r$) in $G$, incident with $e$  (incident with $E_e$), i.e., $G$ is $\Theta$-Las Vergnas.
\end{thm}

\begin{proof}
By the definition of Mandel OM, there exists a single-element extension (an expansion) of $G$ in general position such that the two halfspaces in the obtained OM are Euclidean AOMs. Both halfspaces correspond to intersecting subsets of covectors of $G$, say $G_1$ and $G_2$. Speaking of OMs as graphs this means that $G$ is expanded along $G_1$ and $G_2$. Let $E_e\in\mathcal{E}$ be a $\Theta$-class of $G$. We prove the slightly stronger assertion (than asserted in the theorem) that there exists a simplicial tope $v$ (i.e. of degree $r$) in $G_1 \setminus G_2$ incident with $E_e$, for any chosen $G_1,G_2$, making $G$ Mandel.
We will proceed by induction on the size of $G$ and distinguish two cases:

\begin{case}\label{case:noproduct}
 $G$ is not the direct sum (the Cartesian product) with factor $K_2$, where the factor $K_2$ corresponds to element $e$ ($\Theta$-class $E_e$).
\end{case} 

% Orient the lines of $G_1$ with respect to $E_e$. Without loss of generality assume that there is a
% line in $G_1$ that has at least one antipodal subgraph of rank~$r-1$ completely in $E_e^-$. In fact this can be assumed for the following reasons.

By Observation~\ref{obs:coloop} and the antipodality, $G$ has at least two cocircuits $X_1,X_2$ (maximal antipodal subgraphs) with the property $(X_1)_e,(X_2)_e \neq 0$ (not crossed by $E_e$).  Then at least one lies in $G_1$. Without loss of generality assume that it is in $E_e^-$, otherwise reorient $E_e$. This is, there is a
line in $G_1$ that has at least one cocircuit $X_1$ (an antipodal subgraph of rank $r-1$) with $(X_1)_e = -$ (in $E_e^-$).

Orient the lines of $G_1$ with respect to $E_e$. Note that every line of $G_1$ is a subline of a cycle in $G^*$ that either has all its cocircuits $X$ with $X_e=0$ (i.e. maximal proper antipodal subgraphs crossed by $E_e$) or exactly two of them. Since the expansion according to $G_1$ and $G_2$ is in general position this implies that every line of $G_1$ is either crossed by $E_e$ in exactly one cocircuit, or in all of them. Thus, all the lines are oriented, except the ones with all its cociruits $X$ (antipodal subgraphs) having $X_e=0$ (being crossed by $E_e$).

By the definition of Euclideaness, the orientation is strictly acyclic, hence we can find in $G_1\cap E_e^-$ a cocircuit (i.e. an antipodal subgraph of rank~$r-1$) $X$ such that each of its out-neighbors $Y$ in the cocircuit graph has $Y_e =0$  (i.e. is intersected by $E_e$). Let $X^0$ be the set of all elements ($\Theta$-classes $E_f$), such that $X_f=0$ (that cross $X$). Let $H$ be the OM-deletion (a partial cubes contraction) of $G$ along all its elements ($\Theta$-classes) besides $e$ ($E_e$) and the ones in $X^0$. Let $H_1, H_2$ be the respective images of $G_1, G_2$ in $H$. Then by Observation \ref{lem:exp_antipodal}, each covector of $H_1$ and $H_2$ is an image of a covector in $G_1$ or $G_2$, respectively. In graph theoretical language, this means that all the antipodal subgraphs of $H$ lie completely in $H_1$ or in $H_2$. This proves that $H_1$ and $H_2$ define an expansion in general position of $H$. Moreover, also the orientation of the lines in $H_1$ are inherited from the orientation of lines in $G_1$. Hence
the orientation of $H$ with respect to $E_e$ is strictly acyclic as well.  This proves that $H$ is Mandel by the expansion in general position according to $H_1$ and $H_2$.

By definition, $H$ is obtained by deleting elements $f$ (contracting $\Theta$-classes), for which  $X_f\neq 0$ (not crossing $X$). It is not hard to see, that the sub-lattice of covectors $Y$ with $X\leq Y$ has a maximal chain of length $r-1$ which is unaffected by the deletion. Thus, the rank of $H$ is $r$ since it properly contains $X$ of rank~$r-1$. In graph theoretical sense, this also follows from the gatedness of antipodal subgraphs in an OM implying that the structure of the antipodal subgraph corresponding to $X$ is not affected by any of the contractions. We will denote by $X'$ the cocircuit obtained from $X$ by the deletion.
%
% With a slight abuse of notation we can say that also $H$ contains $X$ as a cocircuit (i.e. antipodal subgraph).

Let $\overline{X'}$ be the antipode of $X'$ in $H$. Then $\overline{X'}$ is also a cocircuit of $H$ and only $e$ separates $X'$ and $\overline{X'}$. As graphs they are disjoint antipodal subgraph connected by edges in  $E_e$. Thus $H\cong H' \oplus K_2$, where $H'$ consist of all $Y$ with $X\leq Y$ in $H$. In the graph language, $H \cong A \, \square \, K_2$, where $A$ is the antipodal subgraph corresponding to $X'$. Since we are in the Case~\ref{case:noproduct}, this gives that $H$ is strictly smaller than $G$.

By the induction assumption, $H$ has a tope $v'$ of degree $r$ in $H_1 \setminus H_2$, incident with $e$ (with $E_e$). In fact, since $H\cong H' \oplus K_2$ (or $H\cong H' \square K_2$) with the $K_2$ factor corresponding to $e$ ($E_e$), the latter holds for all the topes in $H$. By Lemma~\ref{lem:simplicialvertex}, there is a set $\mathcal{B}$ of $r$ cocircuits and a set $F$ of $r$ $\Theta$-classes, such that for each $E_g \in F$ there is a corresponding cocircuit $Y$ with $Y_f=0$ for $f \in F\setminus\{E_g\}$, $Y_g \neq 0$ and $Y \leq v'$. Note that $E_e \in F$. In other words, there are $r$ antipodal subgraphs of rank $r-1$ incident with $v'$ such that $v'$ has degree $r-1$ in each member. Since a tope of degree $r$ cannot have more than $r$ covectors $Y$ with $Y \leq v'$, we have $X'\in \mathcal{B}$. By assumption $v'\in H_1 \setminus H_2$  and since $H_1, H_2$ define an expansion in general position, all the members of $\mathcal{B}$ are in $H_1$. By Observation \ref{lem:exp_antipodal}, there is a set $\mathcal{C}$ of $r$ cocircuits of $G$ together with the same set $F$ of $r$ $\Theta$-classes, such that each member of $\mathcal{C}$ corresponds to to a member of $\mathcal{B}$ in $H$ (corresponding antipodal subgraphs of $\mathcal{C}$ contract to corresponding antipodal subgraphs of $\mathcal{B}$). Moreover, since the members of $\mathcal{B}$ are in $H_1$, the cocircuits in $\mathcal{C}$ are in $G_1$. Clearly, $X\in\mathcal{C}$. 

Consider a tope $v \in G_1$ such that $X\leq v$ ($v$ is an element of the antipodal subgraph corresponding to $X$) and the OM-deletion (partial cube contraction) maps $v$ to $v'$. To prove that $v$ has degree $r$ in $G$, by Lemma~\ref{lem:simplicialvertex}, it suffices to prove that for each cocircuit $Y$ in $\mathcal{C}$ it holds $ Y\leq v$ (the antipodal subgraphs in $\mathcal{C}$ include $v$). Let $Y \in  \mathcal{C}\setminus \{X\}$, and $Y'$ the corresponding cocircuit (antipodal subgraph) in $\mathcal{B}$. Then $Y'$ and $X'$ are adjacent in $H^*$, hence $X$ and $Y$ lie on a line in $G_1$. Since $Y'_e= 0$ also $Y_e=0$, i.e. the corresponding antipodal subgraphs are crossed by $E_e$. Thus the line they are both on is oriented from $X$ to $Y$. Since $X$ was chosen to have all its out-neighbors crossed by $E_e$, $Y$ must be adjacent to it. This implies that for each $E_f \notin X^0 \cup \{E_e\}$, $Y_f \neq -X_f$. In particular, since $X \leq v$ and by assumption $Y' \leq v'$, the former implies that $Y \leq v$. In the tope graphs this is seen as $v$ being in the intersection of the antipodal subgraphs corresponding to $X$ and $Y$ since $v'$ is the intersection of the antipodal subgraphs corresponding to $X'$ and $Y'$.

%the graphs in $\mathcal{C}$ are incident with $v$. Let $D \in \mathcal{C}\setminus \{A\}$, and $D'$ the corresponding graph in $\mathcal{B}$. Since $D'$ and $A$ intersect in a rank~$r-2$ antipodal subgraph and are both in $H_1$, then $A$ and $D$ lie on a line in $G_1$. Moreover, since $D'$ is crossed by $E_e$, so is $D$. Thus, this line is oriented from $A$ towards $D$, thus by the choice of $A$ they are adjacent in $G^*$, and in particular intersect in a rank~$r-2$ antipodal subgraph. Moreover this subgraph must be the subgraph that contracts to the intersection of $D'$ and $A$. Hence, $v$ is in the intersection, thus in $D$. 

We have found a tope $v$ that has degree $r$. Since $e$ (or $E_e$) is in $F$, it is incident with it. By construction it lies in $G_1 \setminus G_2$.

\begin{case}
 $G$ is the direct sum (the Cartesian product) with factor $K_2$, where the factor $K_2$ corresponds to element $e$ ($\Theta$-class $E_e$).
%$G$ is the direct sum $G' \, \oplus \, K_2$ with factor $K_2$ corresponding to $E_e$.  
\end{case}

If $G$ has an element $f$ (a $\Theta$-class $E_f$) such that $G$ is not the direct sum with factor $K_2$ corresponding to it, then by Case~\ref{case:noproduct}, $G$ has a simplicial tope $v$ (of degree $r$) in $G_1 - G_2$. Moreover, in this case all the vertices of $G$ are incident with $E_e$, in particular also $v$ is.

If all of the elements ($\Theta$-classes) of $G$ correspond to factors $K_2$, then $G$ is a hypercube and all its vertices are simplicial and incident to $E_e$.
\end{proof}

OMs with a $\Theta$-class not incident to a simplicial vertex have been found of different sizes~\cite{Ric-93, Bok-01, Tra-04}. We conclude:

\begin{cor}\label{cor:Mandelwaswrong}
 There exist OMs that are not Mandel. The smallest known such OM has $598$ vertices (topes), isometric dimension $13$ (elements) and is uniform of rank~$4$. See Figure~\ref{fig:tracyhall}. 
\end{cor}

% 
% \begin{figure}[ht]
% \centering
% \includegraphics[width=.8\textwidth]{TH.png}
% \caption{A half-space $E_e^+$ of a non-Mandel OM $G$. The bold graph (which is itself a rank~3 OM) is induced on the vertices incident with $E_e$ in $G$, where $E_e$ is the $\Theta$-class witnessing that $G$ is not $\Theta$-Las Vergnas.}
% \label{fig:tracyhall}
% \end{figure}

\section{Corners and corner peelings}\label{sec:sec5corners}

In~\cite{Cha-18} the corners of lopsided sets (LOPs) have been defined as vertices of degree $r$ that lie in a unique maximal hypercube subgraph of dimension $r$. This has a nice correspondence to the results from Section \ref{sec:mandel}:
If $G$ is a LOP and also an AOM, i.e, a halfspace $E^+_e$ of an OM $G'$, then $G$ has a corner if and only if $G'$ has a simplicial vertex incident to $E_e$. By the examples from~\cite{Cha-18,Tra-04,Ban-96} this proves that there are LOPs without corners. This was first observed in~\cite{Cha-18}, where it is translated to an important counter example in computational learning theory.

In the present section we introduce corners and corner peelings for general COMs. The first subsection is concerned with the first definitions and results, and in particular contains a proof for existence of corner peelings of realizable COMs. The second subsection contains corner peelings for COMs of rank~$2$ and hypercellular graphs. Since the results in this section are more concerned with graphs, we use the graph theoretic language and use covectors only when needed.

\subsection{First definitions and basic results}
We will approach our general definition of corner of a COM, that generalizes corners on LOPs and has strong connections with simplicial vertices in OMs.
The intuitive idea of a corner in a COM, is a set of vertices whose removal gives a new (maximal) COM. As a matter of fact it is convenient for us to first define this remaining object and moreover within an OM.

Recall the definition of an expansion in general position from Section~\ref{sec:prel}.
We will say that the subgraph $T$ of an OM $H$ is a \emph{chunk} of $H$, if $H$ admits an expansion in general position $H_1, H_2$, such that $T=H_1$. We call the complement $C=H\setminus H_1$ a \emph{corner} of $H$. In the case that $H$ has rank 1, i.e. $H$ is isomorphic to an edge $K_2$, then a corner is simply a vertex of $H$.

This definition extends to COMs by setting $C$ to be a \emph{corner} of a COM $G$ if $C$ is contained in a unique maximal antipodal subgraph $H$ and $C$ is a corner of $H$. 
We need two more helpful observation:

\begin{lem}\label{lem:isometric}
 If $G'$ is an isometric subgraph of a COM $G$ such that the antipodal subgraphs of $G'$ are antipodal subgraphs of $G$, then $G'$ is a COM. 
\end{lem}
\begin{proof}
 By Theorem \ref{thm:char} all antipodal subgraphs of $G$ are gated, but since $G'$ is an isometric subgraph and it has no new antipodal subgraph also the antipodal subgraphs of $G'$ are gated. Thus, by Theorem \ref{thm:char} $G'$  is a COM.  
\end{proof}

We are now ready to prove that chunks and corners as we defined them achieve what we wanted.  This proof uses the correspondence between sign-vectors and convex subgraphs as introduced in Section~\ref{sec:prel}.

\begin{lem}\label{lem:cornerisweakcorner+peeling}
 If $C$ is a corner of a COM $G$, then the chunk $G\setminus C$ is an inclusion maximal proper isometric subgraph of $G$ that is a COM. 
\end{lem}
\begin{proof}
Let us first consider the case where $G=H$ is an OM.
 Let $T, -T$ be an expansion of $H$ in general position, i.e., $-T$ is the set of antipodes of $T$. Since expansions in general position are OM-expansions, $T$ is a halfspace of an OM. Thus, $T$ is a COM --- even an AOM. 
This proves that $T$ is a sub-COM. Assume that it is not maximal and let $R\supseteq T$ be a COM contained in $H$. Let $X\subseteq R$ be an antipodal subgraph of $H$ that is not completely in $T$, and  is maximal with this property. Since the expansion is in general position, it holds $X \subseteq -T$ and $-X\in T \subseteq R$. Let $E_e\in S(X,-X)$, i.e., $E_e$ separates $X$ from $-X$. Such $E_e$ clearly exists, since $R$ is a proper sub-COM of $H$, thus $X\neq H$.

Considering $R$ as a COM we apply (SE) to $X,-X$ with respect to $E_e$ in order to obtain $Z\subset R$ that is crossed by $E_e$.  Note that $X$ and $-X$ are crossed by the same set of $\Theta$-classes $X^0$. By (SE) the set $Z^0$ of $\Theta$-classes crossing $Z$ strictly contains $X^0$. Thus, if $S(Z,X)=\emptyset$, then $Z$ is an antipodal subgraph containing $X$, i.e., $X$ was not a maximal antipodal subgraph of $R$. Let otherwise $E_f\in S(Z,X)$. Apply (SE) to $X,Z$ with respect to $E_f$ in order to obtain $Z'\subset R$ which is crossed by $E_f$. Since $Z^0\supsetneq X^0$, we have $Z'^0\supsetneq X^0$ and furthermore $S(Z',X)\subsetneq S(Z,X)$. Proceeding this way, we will eventually obtain an $\widetilde{Z}\subseteq R$ with $\widetilde{Z}^0\supseteq X^0$ and $S(\widetilde{Z},X)=\emptyset$. Thus, $\widetilde{Z}\in R$ is an antipodal subgraph containing $X$. By the choice of $X$, $\widetilde{Z}$ is not completely in $T$. This violates the assumption that $X$ was maximal. Thus, $R=T$.
 
 Now, let $G$ be a COM that is not an OM and let $H$ be the unique maximal antipodal subgraph of $G$ containing $C$. By the above $T=H\setminus C$ is an isometric subgraph of $H$ and a COM. Now, it follows that $G\setminus C$ is an isometric subgraph of $G$. Namely, since no vertex of $C$ is adjacent to a vertex of $G\setminus H$ and $T$ being an isometric subgraph of $H$, all shortest paths in $G$ through $C$, can be replaced by shortest paths through $T$. 
 Finally, Lemma~\ref{lem:isometric} implies that $G\setminus C$  is a COM. {Maximality follows from the first paragraph.}
\end{proof}

%We have already argued \comment{maybe in the intro} that simplicial vertices, i.e., vertices with degree equal to the rank in pure LOPs are exactly the corners of pure LOPs\comment{define maybe in the intro}. 
Recall that simplicial vertices in LOPs are called corners.
Before providing further central properties of corners in COMs, let us see that we indeed generalize corners of LOPs.

\begin{prop}\label{prop:cornersofLOPs}
 A subset $C$ of the vertices of a LOP $G$ is a corner if and only if $C=\{v\}$ is contained in a unique maximal cube of $G$.% if and only if $C=\{v\}$ such that $G\setminus\{v\}$ is a LOP.
\end{prop}
\begin{proof}
OMs that are LOPs are cubes, so let $v$ be a vertex of ${Q}_n$. The expansion with $H_1={Q}_n\setminus\{v\}$ and $H_2={Q}_n\setminus\{-v\}$ clearly is antipodal. Moreover, every proper antipodal subgraph of ${Q}_n$ is contained in either $H_1$ or $H_2$. Thus, this expansion is in general position. Consequently $v$ is a corner or ${Q}_n$. Since chunks are maximal sub-COMs, by Lemma \ref{lem:cornerisweakcorner+peeling}, single vertices are precisely the corners of ${Q}_n$. 

If now $v$ is a corner of a LOP $G$, then by the definition of corners of COMs and the fact that in a LOP all proper antipodal subgraphs are cubes, $v$ is contained in a unique maximal cube of $G$.

Conversely if $v$ is a vertex of a LOP contained in a unique cube, then this cube is also the unique maximal antipodal subgraph of the LOP, since in a LOP all proper antipodal subgraphs are cubes. Thus $v$ is a corner of the LOP.
\end{proof}

Note that, as mentioned earlier, every OM admits an expansion in general position, see~\cite[Lemma 1.7]{San-02} or~\cite[Proposition 7.2.2]{bjvestwhzi-93}. This yields directly from the definition:
\begin{obs}
 Every OM has a corner.
\end{obs}

Note however that COMs do not always have corners, e.g., with Proposition~\ref{prop:cornersofLOPs} one sees that the AOMs obtained from the UOMs with a mutation-free element have no corner.

% \begin{lem}\label{lem:COMcorners}
%  If $G$ is a COM and $C$ a corner, then $G\setminus C$ is a COM. Moreover, removing preserves the property of being LOP and also being rank at most $r$.
% \end{lem}
% \begin{proof}
% Let $H$ be the unique maximal antipodal subgraph of $G$ containing $C$. By definition $T=H\setminus C$ is an isometric subgraph of $H$ and a COM. Now, it follows that $G\setminus C$ is an isometric subgraph of $G$. Namely, since no vertex of $C$ is adjacent to a vertex of $G\setminus H$ and $T$ being an isometric subgraph of $H$, all shortest paths in $G$ through $C$, can be replaced by shortest paths through $T$. 
%  Finally, Lemma~\ref{lem:isometric} implies that $G\setminus C$  is a COM.
%  
%  Notice, that by Corollary~\ref{cor:specialcorners} removing a corner in a LOP $G$, just is removing a vertex $v$ of degree $d$ from a $Q_n$. This does not create new antipodal subgraphs and since $G$ being a LOP is equivalent to all antipodal subgraphs being hypercubes, also $G\setminus\{v\}$ is LOP.
%  
%  Lastly, clearly removing a set of vertices from $G$ cannot create a $Q_r$ minor, i.e., the rank can only decrease by removing corners.
% \end{proof}
% 

Lemma~\ref{lem:cornerisweakcorner+peeling} yields the following natural definition.
A \emph{corner peeling} in a COM $G$ is an ordered partition $C_1, \ldots, C_k$ of its vertices, such that $C_i$ is a corner in $G - \{C_1,\ldots, C_{i-1}\}$.
%A subtelty of the above definition is that, when we say that COMs from a certain class have corner peelings, this does not mean that all intermediate COMs are in the same class. Proposition~\ref{prop:cornersofLOPs} however yields that this is the case for LOPs. Another, instance is the class of realizable COMs.
In the following we generalize a results from \cite{Tra-04} for realizable LOPs. 

% A \emph{chunk} $T$ of an OM $H$ is an inclusion maximal proper isometric subgraph of $H$, that is a COM. We make the exception that if $H=K_1$, then $T=H$. A \emph{corner} of a COM $G$ is a set of vertices $C$ that is contained in a unique maximal antipodal subgraph $H$ of $G$ and $H\setminus C$ is a chunk of $H$. 
% 
% A chunk $T$ (and the corresponding corner $C$) is \emph{strong} if either $H=K_1=T$ or there is $e\in\Theta(H)$ such that $\rho_{e^-}(H)\subseteq T$ and $\zeta_e(T)$ is a strong chunk of $\zeta_e(H)$.
\begin{prop}\label{prop:realizable}
Every realizable COM has a corner peeling.
\end{prop}
\begin{proof}
 We show that a realizable COM $G$ has a realizable chunk $T$. Represent $G$ as a central hyperplane arrangement $\mathcal{H}$ in an Euclidean space intersected with an open polyhedron $P$ given by open halfspaces $\mathcal{O}$. Without loss of generality we can assume that the supporting hyperplanes of the halfspaces in $\mathcal{O}$ are in general position with respect to the hyperplanes in $\mathcal{H}$. 
We shall call points in the Euclidean space, that can be obtained as intersection of subset of hyperplanes $\mathcal{H}$ minimal dimensional cells. It follows from the correspondence between antipodal subgraphs and covectors of a COM \cite[Theorem 4.9]{Kna-17}, that topes (chambers) surrounding minimal dimensional cells correspond to antipodal subgraphs of $G$.
 
 Now, take some halfspace $O\in\mathcal{O}$ and push it into $P$ until it contains the first minimal dimensional cell $C$ of $\mathcal{H}$. The obtained realizable COM $T$ is a chunk of $G$, because restricting the antipodal subgraph (an OM) corresponding to the cell $C$ with respect to $O$ is taking a chunk of $C$, while no other cells of $G$ are affected and the resulting graph $T$ is a COM.
\end{proof}

In~\cite[Conjecture 2]{Ban-18} it was conjectured that all \emph{locally realizable} COMs, i.e., those whose antipodal subgraphs are realizable OMs, are realizable. Proposition~\ref{prop:realizable} yields a disproof of this conjecture, since all antipodal subgraphs of a LOP are hypercubes, i.e., LOPs are locally realizable, but by the example in Figure~\ref{fig:tracyhall} and others there are LOPs that do not have corner peelings. Thus, they cannot be realizable. 

\begin{rem}~\label{rem:notrealizable}
 There are locally realizable COMs, that are not realizable.
\end{rem}

Indeed , LOPs are even \emph{zonotopally realizable}, i.e., one can choose realizations for all cells such that common intersections are isometric.
It remains open whether every locally realizable COM is zonotopally realizable, see~\cite[Question 1]{Ban-18}. 

\subsection{Corners and corner peelings in further classes}\label{sec:corners}

In this section we consider the question of the existence of corners and corner peelings in various classes of graphs. By Proposition~\ref{prop:cornersofLOPs} simplicial vertices in LOPs are corners. Thus, Theorem~\ref{thm:simplical_on_theta} yields:

\begin{cor}\label{cor:AOMcorner}
Every halfspace of a Mandel UOM  has a corner.
\end{cor}

In the following we focus on COMs of rank~$2$ and hypercellular graphs. In both these proofs we use the zone graph $\zeta_f(G)$ of a partial cube $G$ with respect to a $\Theta$-class $f$, see~\cite{Kla-12}. In general the zone graph is not a partial cube, but indeed a characterization of COMs from~\cite{Kna-17} (generalizing a result of Handa for OMs~\cite{ha-90}) allows the following definition in COMs. Let $G$ be a COM, and $F$ a subset of its $\Theta$-classes. The \emph{zone graph} $\zeta_F(G)$ is the graph obtained from $G$, whose vertices are the minimal antipodal subgraphs of $G$ that are crossed by all the classes in $F$. It turns out that all such antipodal subgraphs have the same rank, say $r$. Two such antipodal subgraphs are connected in $\zeta_F(G)$ if they lie in a common antipodal subgraph of $G$ of rank~$r+1$. The above mentioned characterizing property of COMs is that $\zeta_F(G)$ is always a COM. In the standard language of OMs, zone graphs are known as contractions of OMs. The zone graph operation will be used frequently in Section~\ref{sec:corners}.Section~\ref{sec:prel}. We start with some necessary observations on cocircuit graphs of COMs.

% In this section we consider the question of the existence of corners and corner peelings in various classes of graphs. We will focus on COMs of rank~$2$ and hypercellular graphs. In both these proofs we will make use of the zone graph of a partial cube, as introduced in Section~\ref{sec:prel}.
% 
% However, first notice that since by Proposition~\ref{prop:cornersofLOPs} simplicial vertices in LOPs are corners, Theorem~\ref{thm:simplical_on_theta} yields:
% 
% \begin{cor}\label{cor:AOMcorner}
% Every halfspace of a uniform Mandel OM $G$ has a corner.
% \end{cor}

\subsubsection*{Cocircuit graphs of COMs}
In the following we generalize the concept of orientation of the cocircuit graph introduced in Section~\ref{sec:mandel} from AOMs to general COMs. 

\begin{lem}\label{lem:rankofcell}
 If $G$ is a COM and a hypercube $Q_r$ a minor of $G$, then there is an antipodal subgraph $H$ of $G$ that has a $Q_r$  minor. In particular, the rank of a maximal antipodal subgraph of a COM $G$ is the rank of $G$.
\end{lem}
\begin{proof}
Since $Q_r$ is antipodal, by Lemma \ref{lem:exp_antipodal}, there exist an antipodal subgraph $H$ of $G$ that contracts to it. Then $H$ is the desired subgraph.
\end{proof}

We define the \emph{cocircuit graph} of a non-antipodal rank~$r$ COM as the graph whose vertices are the rank~$r$ antipodal subgraphs and two vertices are adjacent if they intersect in a rank~$r-1$ antipodal subgraph. By Lemma~\ref{lem:rankofcell} the vertices of the cocircuit graph of a non-antipodal COM $G$ correspond  to the maximal antipodal subgraphs of $G$. The cocircuit graph of a COM can be fully disconnected hence we limit ourselves to COMs having all its maximal antipodal subgraphs of the same rank with $G^*$ connected. We call them \emph{pure} COMs. Note that AOMs are pure COMs.

%First we need to better understand the connection between the cocircuit graph of a pure COM $G$ and the zone graphs of $G$. Let $A_1, A_2$ be an edge in $G^*$, and $F$ be the set of $\Theta$-classes crossing $A_1\cap A_2$. Let $\zeta_F(G)$ be the zone graph of $G$ with respect to the $\Theta$ classes in $F$. Then it is not hard to see that $A_1\cap A_2$ corresponds to a vertex in in $\zeta_F(G)$, while both $A_1, A_2$ correspond to edges in $\zeta_F(G)$. More generally, all the vertices in $G^*$ that are crossed by the classes in $F$ correspond to edges in $\zeta_F(G)$ and in fact induce a graph isomorphic to the line graph of $\zeta_F(G)$. In this sense $G^*$ can be seen as an edge disjoint union of the line graphs of the (certain) zone graphs.

Let $G$ be a pure COM, $\{A_1, A_2\}$ be an edge in $G^*$ and $F$ be the set of $\Theta$-classes crossing $A_1\cap A_2$. We have seen in Section \ref{sec:mandel} that if $G$ is an AOM, then the maximal proper antipodal subgraphs of $G$ crossed by $\Theta$-classes in $F$ induce a path of $G^*$ which we called a line.
The following lemma is a generalization of the latter and of general interest with respect to cocircuit graphs of COMs, even if we will use it only in the case of rank~$2$.
\begin{lem}\label{lem:linetree}
Let $G$ be a COM that is not an OM, $\{A_1, A_2\}$ be an edge in $G^*$, and $F$ be the set of $\Theta$-classes crossing $A_1\cap A_2$. Then the maximal proper antipodal subgraphs of $G$ crossed by $\Theta$-classes in $F$ induce a subgraph of $G^*$ isomorphic to the line graph of a tree.
\end{lem}
\begin{proof}
Let $G, A_1, A_2, F$ be as stated and $r$ the rank of $G$. Consider the zone-graph $\zeta_F(G)$. Recall that its vertices are antipodal subgraphs of rank~$r-1$ crossed by $\Theta$-classes in $F$ where two subgraphs are adjacent if they lie in a common rank~$r$ antipodal subgraph. Since $\zeta_F(G)$ is a COM, see e.g.~\cite{Kna-17} and has rank~$1$, we have that $\zeta_F(G)$ is a tree. By definition, the maximal proper antipodal subgraphs of $G$ crossed by $\Theta$-classes in $F$ correspond to edges of $\zeta_F(G)$, with two such edges connected if they share a vertex. Hence they form a subgraph of $G^*$ that is isomorphic to the line graph of $\zeta_F(G)$.
\end{proof}

%Mandel proved that every rank 2 AOM is Euclidean. In the following we want to generalize this fact to pure COMs. In \cite{phd_thesis} it was proved that a COM has rank 2 if and only if all its zone graphs are isomorphic to trees, a class of graph also known as tree-zone graphs. Notice the correspondence between the zone graphs and cocircuit graph of a rank 2 COM. In fact the line graphs of zone graphs can be seen as subgraphs of the cocircuit graph, more precisely the cocircuit graph is a edge disjoint union of the line graphs of zone graphs. This implies that we can extend the concept of Euclidean AOMs to rank 2 pure COMs.

Lemma~\ref{lem:linetree} implies that $G^*$ can be seen as the edge disjoint union of line graphs of trees. We can use this to orient edges of $G^*$. Similarly as in the settings of AOMs, we will call a \emph{line} in $G$ a maximal path $L=A_1,\ldots, A_n$ in the cocircuit graph $G^*$ such that $A_{i-1} \cap A_i$ is the set of 
antipodes of $A_i \cap A_{i+1}$ with respect to $A_i$. Let now $E_e\in\mathcal{E}$ be a $\Theta$-class of 
$G$. Similarly as before we say that $E_e$ crosses a line $L$ of $G^*$ if there exists $A_i$ on $\ell$ that is crossed by $E_e$ but $A_{i-1} \cap A_{i}$ or $A_i \cap A_{i+1}$ is not crossed by it. If $A_i$ exists, it is unique.
\emph{The orientation of $L$ with respect to $E_e$} is the orientation of the path $L$ in $G^*$ from $E_e^-$ to $E_e^+$ if $E_e$ crosses $L$ and not orienting the edges of $L$ otherwise. Notice that in this way we can orient the edges of $G^*$ with respect to $E_e$ by orienting all the lines simultaneously. %Then each line is part of some line graph of a tree with a unique $A_i$ crossed by $E_e$ and all the edges of this line graph of a tree that point 
The orientation of each edge (if it is oriented) is well defined: If $\{A_j, A_{j+1}\} \in E_e^-$ is an edge in a line graph of a tree that is crossed by $E_e$ in $A_i$ and $A_{j+1}$ is closer to $A_i$ than $A_{j}$ is, then $\{A_j, A_{j+1}\}$ is oriented from  $A_j$ to $A_{j+1}$. Similarly if $\{A_j, A_{j+1}\} \in E_e^+$  and $A_{j}$ is closer to $A_i$ than $A_{j+1}$ is, then $\{A_j, A_{j+1}\}$ is oriented from  $A_j$ to $A_{j+1}$. Furthermore, $\{A_j, A_k\}$ is not oriented if $A_j, A_k$ are at the same distance to $A_i$. See Figure~\ref{fig:examplelines} for an illustration.

\begin{figure}[ht]
\centering
\includegraphics[width=.8\textwidth]{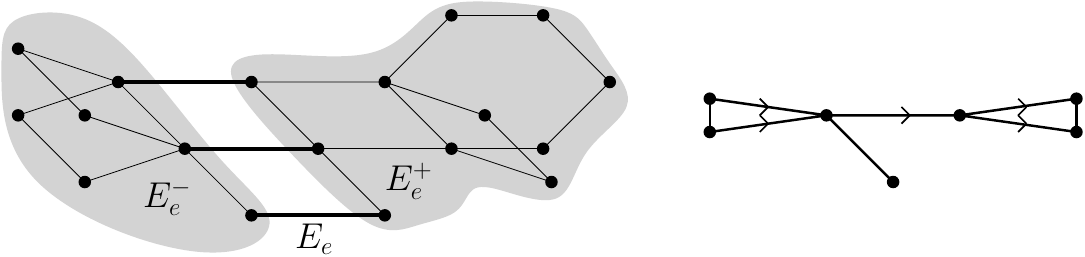}
\caption{A pure rank 2 COM and its cocircuits graph  oriented with respect to $E_e$. }
\label{fig:examplelines}
\end{figure}

\subsubsection*{COMs of rank~$2$}
Mandel proved that every AOM of rank~$2$ is Euclidean, which by Corollary~\ref{cor:AOMcorner} implies that every rank 2 halfspace of a UOM has a corner. We generalize this result.

Let us first consider what corners in rank 2 COMs are. Up to isomorphism the only rank 2 OMs are even cycles. An expansion in general position of an even cycle $G= C_{2n}$ is given by $G_1, G_2=-G_1$, where $G_1$ consists of an induced path on $n+1$ vertices. Hence a corner in a rank 2 COM consist of $n-1$ vertices inducing a path, included in a unique antipodal $C_{2n}$. For example, the COM in Figure \ref{fig:examplelines} has 11 corners. Those contained in a square are single vertices and the ones contained in the $C_6$ are paths with two vertices.
\begin{prop}\label{prop:rank2corner}
Let $G$ be a pure COM of rank 2 and $E_e$ a $\Theta$-class of $G$. Then $G$ has a corner in $E_e^+$ and  in $E_e^-$.
\end{prop}
\begin{proof}
Let $E_e\in\mathcal{E}$ be a $\Theta$-class of a pure rank 2 COM $G$. We shall prove that $G$ has a corner in $E_e^+$, by symmetry it follows that it has one in $E_e^-$ as well. Without loss of generality assume that there is no $\Theta$-class $E_f$ completely contained in $E_e^+$, otherwise switch $E_e^+$ with $E_f^+$ or $E_f^-$ depending on which is entirely in $E_e^+$. Orient the edges of $G^*$ with respect to $E_e$. 

Since the rank of $G$ is 2, the maximal antipodal subgraphs are even cycles and each line consist of sequence of cycles pairwise crossing in edges from $E_f$, for some $E_f$. We will say that the line \emph{follows} $E_f$. The induced subgraph of $G^*$ of all the maximal antipodal subgraphs crossed by $E_f$ is the line graph of a tree, by Lemma~\ref{lem:linetree}. We will denote it by $G_f^*$. Let $G_f^*$ be crossed by some $E_g$ in $A_i$. Then this splits the vertices of $G_f^* - \{A_i\}$ into the ones lying in $E_g^+$ and the ones lying in $E_g^-$. We denote these by $E_g^+(G_f^*)$ and $E_g^-(G_f^*)$, respectively. 

We shall prove that there is no directed cycle in $G^* \cap E_e^+$ consisting of only directed edges. For the sake of contradiction assume that such a cycle exists and take one that is the union of as few parts of lines as possible. Let the cycle be a union of a part of $L_1$, a part of $L_2$,\ldots, and a part of $L_n$. Also denote by $E_{e_1}, \ldots, E_{e_n}$ the respective $\Theta$-classes followed by $L_1, \ldots, L_n$.

Since $L_i$ and $L_{i+1}$ intersect, $L_i$ must be crossed by $E_{e_{i+1}}$ and $L_{i+1}$ must be crossed by $E_{e_{i}}$.
Without loss of generality assume that $L_2$ passes $E_{e_1}$ from $E_{e_1}^-$ to $E_{e_1}^+$, $L_3$ passes $E_{e_2}$ from $E_{e_2}^-$ to $E_{e_2}^+$, \ldots, and $L_1$ passes $E_{e_n}$ from $E_{e_n}^-$ to $E_{e_n}^+$, otherwise reorient  $E_{e_1}, \ldots, E_{e_n}$.

First we show that each $L_i$ passes $E_{e_{i+1}}$ from $E_{e_{i+1}}^+$ to $E_{e_{i+1}}^-$. Assuming otherwise the intersection of $L_{i-1}$ and $L_{i}$ lies in $E_{e_{i+1}}^-$, while the intersection of $L_{i+2}$ and $L_{i+3}$  lies in $E_{e_{i+1}}^+$, by the assumption in the previous paragraph. Then one of the lines $L_{i+4},L_{i+5}, \ldots L_{i-1}$ must pass $E_{e_{i+1}}$ from $E_{e_{i+1}}^+$ to $E_{e_{i+1}}^-$, say $L_j$ passes it. If this passing is in $E_{e_{i}}^+$, then the cycle is not minimal, since one could just replace the lines $L_{i+1},\ldots, L_j$ by the line following $E_{e_{i+1}}$ starting from the intersection of $L_{i}$ and $L_{i+1}$ to the crossing of $E_{e_{i+1}}$ and $L_j$.
In fact such a directed line exists since by assumption the lines following $E_{e_{i+1}}$ pass $E_{i}$ from $E_i^-$ to $E_i^+$ thus the orientation of the shortest path from the intersection of $L_{i}$ and $L_{i+1}$  to the crossing of $L_{j}$ and $E_{e_{i+1}}$ is correct. 

On the other hand, assume the passing is in $E_{e_{i}}^-$. By assumption, the intersection of $L_{i+1}$ and $L_{i+2}$ is in $E_{e_{i}}^+$. Hence one of the lines $L_{i+3},L_{i+4}, \ldots L_{j}$ must pass $E_{e_{i}}$, say $L_{l}$. In particular it must pass it in $E_{e_{i+1}}^+$, by the choice of $L_j$. But then again the cycle is not minimal, since one could just continue on the line following $E_{e_{i}}$ starting from the intersection of $L_{i-1}$ and $L_{i}$ to the crossing of $E_{e_{i}}$ and $L_{l}$. This cannot be.

Now we show that for each $e_i, e_{i+1}$ it holds that $E_{e_{i+1}}^-(G_e^*)  \subset E_{e_{i}}^-(G_e^*)$. Since $L_{i+1}$ passes $E_{e_i}$ from $E_{e_i}^-$ to $E_{e_i}^+$ it follows that $L_{i+1}$ passes $E_e$ in $E_{e_i}^-$. Since $G_e^*$ is the line graph of a tree, this implies that either $E_{e_{i+1}}^+(G_e^*)  \subset E_{e_{i}}^-(G_e^*)$ or $E_{e_{i+1}}^-(G_e^*)  \subset E_{e_{i}}^-(G_e^*)$. On the other hand, $L_{i}$ passes $E_{e_{i+1}}$ from $E_{e_{i+1}}^+$ to 
$E_{e_{i+1}}^-$ hence $L_{i}$ passes $E_e$ in $E_{e_{i+1}}^+$. Again since $G_e^*$ is the line graph of a tree, this implies that $E_{e_{i}}^+(G_e^*)  \subset E_{e_{i+1}}^+(G_e^*)$ or $E_{e_{i}}^-(G_e^*)  \subset E_{e_{i+1}}^+(G_e^*)$. Hence, $E_{e_{i+1}}^-(G_e^*)  \subset E_{e_{i}}^-(G_e^*)$ and $E_{e_{i}}^+(G_e^*)  \subset E_{e_{i + 1}}^+(G_e^*)$.

Inductively 
$E_{e_{n}}^-(G_e^*)  \subset E_{e_{n-1}}^-(G_e^*) \subset \dots \subset E_{e_{1}}^-(G_e^*) \subset E_{e_{n}}^-(G_e^*)$ -- contradiction. This proves that there is no directed cycle in $G^* \cap E_e^+$ consisting of only directed edges.

We can now prove that $G$ has a corner in $E_e^+$. First, assume that  $G^* \cap E_e^+$ is non-empty and let $A \in G^* \cap E_e^+$ be a maximal antipodal subgraph, i.e., an even cycle, that has no out-edges in $G^*$. By the choice of $E_e$, each line $L$ that passes $A$ is crossed by $E_e$. We now analyze how lines pass $A$. Let $L_1, L_2$ be lines passing $A$, following  $E_{f_1},  E_{f_2}$, respectively. Since  $E_{f_2}$ crosses at most one antipodal subgraph of $G_{f_1}^*$, this implies that  $L_1$ and $L_2$ simultaneously pass only $A$. In particular each antipodal subgraph of $G^*_e$ is passed by at most one line passing $A$. Since  $G^*_e$ is the line graph of a tree, its every vertex is a cut vertex. Then each line $L$, passing $A$ and $A_f \in G_e^*$, and following some $E_{f}$, splits $G_e^* - \{A_f\}$ into two connected components, $E_f^+(G_e^*)$ and $E_f^-(G_e^*)$. Thus we can inductively find $L$ such that any other line passing $A$ passes an antipodal subgraph in $G_e^*$ in $E_f^+(G_e^*)$, reorienting $E_f$ if necessary. 

We now show that $A$ includes a corner. Let $A'$ be an antipodal subgraph on $L$ that is a neighbor of $A$. Then $A\cap A'$ corresponds to an edge in $E_f$. Define the set $C$ to include all the vertices of $A$ in $E_f^-$ besides the one vertex lying in $A\cap A'$. Then $C$ is a corner of $A$, we will show that $C$ is a corner of $G$. For the sake of contradiction assume that a vertex $v$ of $C$ lies in a maximal antipodal subgraph $A''$ different from $A$.

We prove that we can choose  $A''$ such that it shares an edge with $A$. Assuming otherwise, since $G$ is a pure COM, there is a path in $G^*$ between $A$ and $A''$. This implies that there is a cycle $C_k$ in $G$ with subpath $v'vv''$, where $v'\in A\setminus A''$ and $v'' \in A''\setminus A$. By \cite[Lemma 13]{Che-16}, the convex cycles span the cycle space in a partial cube. If $A$ is one of the convex cycles spanning $C_k$, then there is a convex cycle incident with $A$ in $v$ and sharing an edge with $A$. If $A$ is not used to span $C_k$, then a convex cycle incident with edge $v'v$ must be used, thus again we have a convex cycle sharing $v$ and an edge with $A$.

We hence assume that $A$ and $A''$ share an edge $g$. By definition of $C$, either $g \in E_f^{-}$, or $g\in E_f$ but not in $A\cap A'$. The latter case implies that $L$ can be extended with $A''$, which cannot be since $A$ in $G^*$ has no out-edges. Moreover, by the choice of $L$, all the other lines passing $A$ pass $E_f$ form $E_f^+$ to $E_f^-$. Then in the former case, some other line passing $A$ can be extended, leading to a contradiction. This implies that $G$ has a corner. 

Finally, consider the option that $G^* \cap E_e^+$ is empty. Since $G_e^*$ is the line graph of a tree, we can pick $A\in G_e^*$ that corresponds to a pendant edge in a tree, i.e. an edge with one endpoint being a leaf. Then it is easily seen that $A$ has a corner in $E_e^+$. This finishes the proof.
\end{proof}

The following is a common generalization of corresponding results for cellular bipartite graphs~\cite{Ban-96} (being exactly rank~$2$ hypercellular graphs, which in turn are COMS~\cite{Che-16}) and LOPs of rank 2~\cite{Cha-18}.

\begin{thm}\label{thm:peeling}
Every rank 2 COM has a corner peeling.
\end{thm}
\begin{proof}
Notice that a rank 2 COM is pure if and only if it is 2 connected. Consider the blocks of 2-connectedness of a rank 2 COM $G$. Then a block corresponding to a leaf in the tree structure of the block graph has 2 corners by Proposition~\ref{prop:rank2corner}. This  implies that $G$ has a corner. Proposition~\ref{lem:cornerisweakcorner+peeling} together with the observation that $G$ minus the corner has rank at most $2$ yield a corner peeling.
\end{proof}

\subsubsection*{Hypercellular graphs}

Hypercellular graphs were introduced as a natural generalization of median graphs, i.e., skeleta of CAT(0) cube complexes in~\cite{Che-16}. They are COMs with many nice properties one of them being that all their antipodal subgraphs are Cartesian products of even cycles and edges, called \emph{cells}. See Figure~\ref{fig:exmpl} for an example.

\begin{figure}[ht]
\centering\includegraphics[width=.5\textwidth]{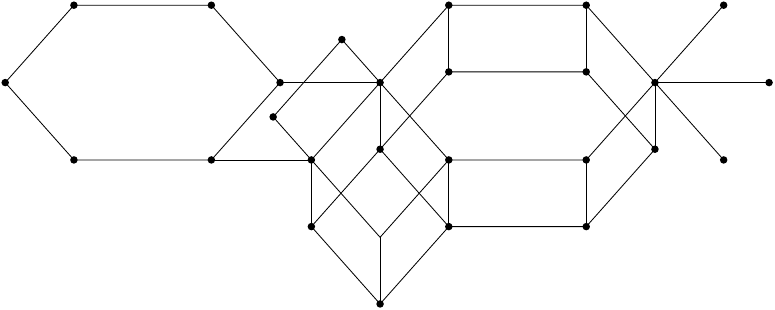}
\caption{A hypercellular graph.}
\label{fig:exmpl}
\end{figure}

More precisely, a partial cube $G$ is \emph{hypercellular} if all its antipodal subgraphs are cells and if three cells of rank~$k$ pairwise intersect in a cell of rank~$k-1$ and altogether share a cell of rank~$k-2$, then all three lie in a common cell, for all $2\leq k\leq r(G)$. See Figure~\ref{fig:3CC} for three rank 2 cells (cycles) pairwise intersecting in rank 1 cells (edges) and sharing a rank 0 cell (vertex) lying in a common rank~3 cell (prism). %Such graphs are a natural generalization of median graphs. 
Since median graphs are realizable COMs, see~\cite{mar-18}, which is also conjectured for hypercellular graphs~\cite{Che-16}, they have corner peelings by Proposition~\ref{prop:realizable}. Here, we prove that hypercellular graphs have a corner peeling, which can be seen as a support for their realizability.

\begin{figure}[ht]
\centering
\includegraphics[width=.5\textwidth]{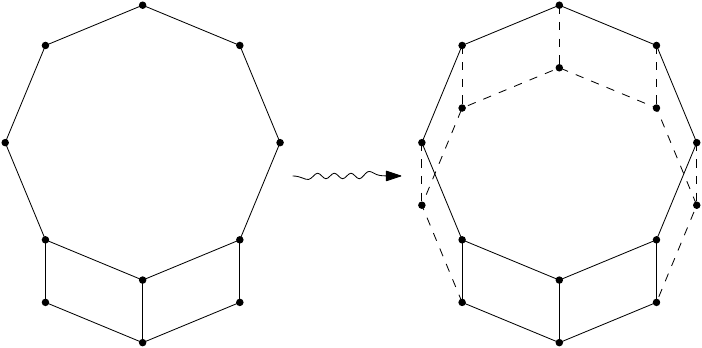}
\caption{Any three cycles of a hypercellular graph that pairwise intersect in an edge and share a vertex lie in a common cell.}
\label{fig:3CC}
\end{figure}

\sloppy
The following lemma determines the structure of corners in hypercellular graphs, since the corners of an edge $K_2$ and an even cycle $C_{2n}$ are simply a vertex and a path $P_{n-1}$, respectively. 
\begin{lem}\label{lem:cornercell}
Let $G = \square_i A_i$ be the Cartesian product of even cycles and edges. Then the corners of $G$ are precisely sets of the form $\square_i D_i \subset G$,  where $D_i$ is a corner of $A_i$ for every $i$. 
\end{lem}
\begin{proof}
Let $D = G_2\setminus G_1$ be a corner of $G$, i.e.~$G_1, G_2$ define an expansion in general position.
Every subset of the form $A_1 \square \ldots \square A_{i-1} \square \{v\} \square A_{i+1} \square \ldots  \square A_n$ is an antipodal subgraph, thus it is either completely in $G_1$ or in $G_2$. We can use the latter to define an expansion in general position of $A_i$ according to $H_1, H_2$, by $v\in H_j$ if $A_1 \square \ldots \square A_{i-1} \square \{v\} \square A_{i+1} \square \ldots  \square A_n \in G_j$, for $j\in\{1,2\}$. In fact each maximal antipodal subgraph $A$ of $A_i$ is either completely in $H_1$ or in $H_2$ (but not both) since 
$A_1 \square \ldots \square A_{i-1} \square A \square A_{i+1} \square \ldots  \square A_n$
is a maximal antipodal of $G$ either completely in $G_1$ or in $G_2$ (but not both). Moreover, since $G_1 = -G_2$ also $H_1= -H_2$.

It remains to prove that $H_1, H_2$ are isometric. Let $v_1, v_2 \in H_1$. If there exists a unique shortest $v_1,v_2$-path $P$ in $A_i$, then pairs of vertices of $A_1 \square \ldots \square A_{i-1} \square \{v_1\} \square A_{i+1} \square \ldots  \square A_n$ and $A_1 \square \ldots \square A_{i-1} \square \{v_2\} \square A_{i+1} \square \ldots  \square A_n$ also have unique shortest paths. This implies that also $A_1 \square \ldots \square A_{i-1} \square P \square A_{i+1} \square \ldots  \square A_n \in G_1$, since $G_1$ is isometric. Thus $P\in H_1$. The only option that $v_1, v_2$ do not have a unique $v_1,v_2$-path is that $A_i$ is a cycle of length $2k$, for $k>2$, and $v_1, v_2$ antipodal vertices in $A_i$. Then for each vertex $u$ of $A_i$ different form $v_1,v_2$ holds that $A_1 \square \ldots \square A_{i-1} \square \{u\} \square A_{i+1} \square \ldots  \square A_n  \in G_2$, thus $u\in H_2$. But then the neighbors of $v_1$ in $A_i$ are in $H_2$ which by the previous case implies that  $v_1 \in H_2$. Similarly $v_2 \in H_2$, which cannot be, since then $G_2 = G$. This proves that $H_1, H_2$ are isometric. In particular, if $A_i$ is a cycle $C_{2n}$, then $H_2, H_2$ are paths $P_{n+1}$, and if $A_i$ is an edge $K_2$, then $H_1, H_2$ are vertices. Thus each $D_i = H_2\setminus H_1$ is a corner.

By definition of the corners $D_j$, it holds $D = (G_2\setminus G_1) \subset  \square_i D_i$, i.e.~$G\setminus\square_i D_i \subset G_1$. We prove that the equality holds. Let $F$ be a set of $\Theta$-classes that cross $\square_i D_i$. Contracting all the $\Theta$-classes in $F$ gives a COM $\pi_F(G)=  \square_i A'_i$ where $A'_i = \pi_F(A_i)$ is a 4-cycle if $A_i$ is a cycle and $A'_i = A_i$ if $A_i$ is an edge. Thus the rank of $\pi_F(G)$ is the same as the rank of $G$. Now if $(G_2\setminus G_1) \subsetneq  \square_i D_i$, it holds $\pi_F(G_1) = \pi_F(G)$, since $\pi_F(\square_i D_i)$ is a vertex. Defining the expansion of $\pi_F(G_1)$ according to $H_1=\pi_F(G_1)=\pi_F(G), H_2 = \pi_F(G_2)=\pi_F(G)$, gives a graph $H$ that has a higher rank than $\pi_F(G)$ and $G$. But $H$ can be obtain as a contraction of the graph $H'$ obtained by expanding $G$ with respect to $G_1$ and $G_2$. Since $G_1, G_2$ define an expansion in general position $H'$ has the same rank as $G$. This is impossible.

We have proved that if $G$ has a corner, then it is of the form $\square_i D_i$. As mentioned, every OM has a corner. By symmetry, every set of vertices of the form $\square_i D_i$ is a corner of $G$.
\end{proof}

We shall need the following property about hypercellular graphs.
\begin{lem}\label{lem:hypercellularzoneclosed}
Every zone graph $\zeta_f(G)$ of a hypercellular graph $G$ is hypercellular.
\end{lem}
\begin{proof}
Every zone graph of the Cartesian product of even cycles and edges is the Cartesian product of even cycles and edges, as it can easily be checked. Let $\zeta_f(G)$ be a zone graph of a hypercellular graph $G$. Then every cell of rank~$r$ in $\zeta_f(G)$ is an image of a cell of rank~$r+1$ in $G$. Hence  for every three rank~$r$ cells pairwise intersecting in rank~$r-1$ cells and sharing a rank~$r-2$ cell from $\zeta_f(G)$, there exist three rank~$r + 1$ cells pairwise intersecting in rank~$r$ cells and sharing a rank~$r-1$ cell in $G$. Additionally the latter three cells lie in a common cell $H$ in $G$. Then the image of $H$ in $\zeta_f(G)$ is a common cell of the three cells from $\zeta_f(G)$.
\end{proof}

Let  $E_e$ be a $\Theta$-class of a COM $G$. As usual, see e.g.~\cite{Ban-18,Che-16}, we call the union of antipodal subgraphs crossed by $E_e$ the \emph{carrier} of  $E_e$. The following is another generalization of the corresponding result for cellular graphs~\cite{Ban-96} and as mentioned above for median graphs.

\begin{thm}\label{thm:hypercellular}
Every hypercellular graph $G$ has a corner peeling.
\end{thm}
\begin{proof}
We prove the assertion by induction on the size of $G$. The technical difficulty of the proof is that removing a corner in a hypercellular graph possibly produces a non-hypercellular graph. Hence we shall prove the above statement for the larger family $\mathcal{F}$ of COMs defined by the following properties:
\begin{enumerate}[(1)]
\item Every antipodal subgraph of $G\in\mathcal{F}$ is a cell.
\item Every carrier of $G\in\mathcal{F}$ is convex.
\item Every zone-graph of $G\in\mathcal{F}$ is in $\mathcal{F}$.
\end{enumerate}

We first prove that hypercellular graph are a part of $\mathcal{F}$. By Lemma \ref{lem:hypercellularzoneclosed} only the first two properties must be checked. Now, (1) holds by definition of hypercellular graphs. Moreover, (2) follows from the fact that for any $\Theta$-class $E_e$ in a hypercellular graph the carrier of $E_e$  is gated \cite[Proposition 7]{Che-16}, thus also convex.
 
We now prove that the graphs in $\mathcal{F}$ have a corner peeling. Let $G \in \mathcal{F}$ and $E_e$ an arbitrary $\Theta$-class in $G$. Since the carrier of $E_e$ is convex the so-called Convexity Lemma \cite{imrich1998convexity} implies that for any edge $g\in E_e^+$ with exactly one endpoint in the carrier its $\Theta$-class $E_g$ does not cross the carrier.
 Now if the union of cells crossed by $E_e$ does not cover the whole $E_e^+$, then for any edge $g$ in $E_e^+$ with exactly one endpoint in the union, one of $E_g^+$ or $E_g^-$ is completely in $E_e^+$. Repeating this argument with $E_g$ one can inductively find a $\Theta$-class $E_f$ with the property that the carrier of $E_f$ completely covers $E_f^+$, without loss of generality.

 Let $\zeta_f(G)$ be the zone graph of $G$ with respect to $E_f$, i.e., the edges of $E_f$ are the vertices of $\zeta_f(G)$ and two such edges are connected if they lie in a common convex cycle. By (3) $\zeta_f(G)$ is in $\mathcal{F}$, thus by induction $\zeta_f(G)$ has a corner $D_f$. By definition there is a maximal antipodal subgraph $A_f$ in $\zeta_f(G)$ such that the corner $D_f$ is completely in $A_f$. Moreover, there exists a unique maximal antipodal subgraph $A$ in $G$ whose zone graph is $A_f$.

We lift the corner $D_f$ from $A_f$ to a corner $D$ of $A$ in the following way. If $E_f$ in $A$ corresponds to an edge factor $K_2$, then $A$ is simply $K_2 \square A_f$. In particular we can define $D= \{v\} \square D_f$ where $v$ is a vertex of $K_2$ in $E_f^+$. By Lemma \ref{lem:cornercell}, this is a corner of $A$. Since $D_f$ lies only in the maximal antipodal graph $A_f$, $D$ lies only in $A$.

Otherwise, assume $E_f$ in $A$ corresponds to a $\Theta$-class of a factor $C_{2k}$ (an even cycle). We can write $A=C_{2k} \square A'$. Then $A_f = K_2 \square A'$ with a corner $D_f = \{v\} \square D'$, where $D'$ is a corner of $A'$ by Lemma \ref{lem:cornercell}. We lift $D_f$ to $D = P_{k-1} \square D'$. Here $P_{k-1}$ is the path in $C_{2k}$ consisting of the vertices in $E_f^+$ apart from the one lying on the edge not corresponding to $v$ in the zone graph. As above since $D_f$ lies only in the maximal antipodal graph $A_f$, $D$ lies only in $A$.

We have proved that $G$ has a corner $D$. To prove that it has a corner peeling it suffice to show that $G \backslash D$ is a graph in $\mathcal{F}$. Since removing a corner does not produce any new antipodal subgraph, all the antipodal subgraphs of $G \backslash D$ are cells, showing (1). The latter holds also for all the zone graphs of $G \backslash D$. To prove that (2) holds for $G \backslash D$ consider a $\Theta$-class $E_e$ of $G \backslash D$. By Lemma \ref{lem:cornerisweakcorner+peeling}, $G \backslash D$ is an isometric subgraph of $G$, i.e.~all the distances between vertices are the same in both graphs. Since the carrier of $E_e$ in $G$ is convex and removing a corner does not produce any new shortest path, the carrier of  $E_e$  is convex in $G \backslash D$. 
%Let $v,u$ be vertices of the carrier of $E_e$. Since the carrier of $E_e$ in $G$ is convex, there exists a shortest $v,u$-path $P$ in $G$ completely in the carrier. If $P$ does not intersect $A$, then $P$ is a shortest $v,u$-path in $G \backslash D$ completely in the carrier of $E_e$. Otherwise, the first and the last vertex of $P$ intersecting $A$ are not in $D$. Thus one can redirect $P$ in $A \backslash D$ to obtain a shortest $v,u$-path in $G \backslash D$ completely in the carrier of $E_e$.
The same argument can be repeated in any zone graph of $G \backslash D$. This finishes the proof.
%By Lemma~\ref{lem:cornerisweakcorner+peeling}, removing $D$ from $G$ gives a COM $G'$ which is not necessarily a hypercellular graph. Nevertheless, the zone graph $\zeta_f(G')$ has a corner peeling\comment{why? is this because it is $\zeta_f(G)$ minus the first corner?}. This implies that we can iteratively derive a sequence of corners in $G$. After the deletion of all such corners it remains exactly $E_f^-$. Since $E_f^-$ is a hypercellular graph, it has a corner peeling, by induction assumption. This implies that $G$ has a corner peeling.
\end{proof}

We have shown corner peelings for COMs of rank~$2$ and hypercellular graphs. A common generalization are Pasch graphs~\cite{Che-94,Che-16}, which form a class of COMs~\cite{Kna-17} that exclude the examples from~\cite{Cha-18,Tra-04,Ban-96}:
\begin{quest}\label{quest:pasch}
 Does every Pasch graph have a corner peeling?
\end{quest}

\section{The minimum degree in antipodal partial cubes}\label{sec:antipodal}
Las Vergnas' conjecture can be seen as a statement about the minimum degree of an OM of given rank. Here we examine the relation of rank and minimum degree in general antipodal partial cubes.
\subsection{Lower bounds}

As stated in Section \ref{sec:intro}, if $G$ is the tope graph of an OM, then $r(G)\leq\delta(G)$, see \cite[Exercise 4.4]{bjvestwhzi-93}.
In general rank~$r$ antipodal partial cubes the minimum degree is not bounded from below by $r$. More precisely:
\begin{prop}\label{prop:nolowerbound}
 For every $r\geq 4$ there is an antipodal partial cube of rank~$r$ and minimum degree $4$. Moreover,  there is an antipodal partial cube of rank~$4$ and minimum degree $3$.
\end{prop}
\begin{proof}
 In~\cite{Kna-17} it is been shown that every partial cube $G$ with $n$ $\Theta$-classes -- thus embeddable in $Q_n$ -- is a convex subgraph of an antipodal partial cube $A_G$. Here, $A_G$ is obtained by replacing in a $Q_{n+3}$ one $Q_n$ by $G$ and its antipodal $Q_n$ by $-G$. It is straight-forward to see that the minimum degree of $A_G$ is $\delta(G)+3$ and that the rank of $A_G$ is at least $n+2$. Indeed, for instance taking $G$ as a path of length $k>1$ we get $\delta(A_G)=4$ and $r(A_G)=k+2$.
 
 Another construction is as follows. Take $Q_n^{-\, -}(i)$, with $1\leq i<n$ and $n\geq 4$, to be the graph obtained from $Q_n$ by removing a vertex $v$, its antipode $-v$ and $i$ neighbors of $-v$. Such a graph is affine and each antipode (in $Q_n$) of the removed neighbors of $-v$ is without the antipode in $Q_n^{-\, -}(i)$, is of degree $n-1$ and of rank~$n-1$. Then construct the antipodal graph taking two antipodal copies of it. Such graph will have minimum degree $n-1$ and rank~$n$. For $n=4$ this gives the second part of the result.
\end{proof}

On the other hand, it is shown in~\cite{polat2018some} that if an antipodal partial cube $G$ has $\delta(G)\leq 2$, then $r(G)=\delta(G)$. This implies that if an antipodal partial cube $G$ has $r(G)\leq 3$, then $r(G)\leq\delta(G)$.

In relation to a question about cubic non-planar partial cubes we ask the following:
\begin{quest}
Are there antipodal partial cubes with minimum degree $3$ and arbitrary rank? 
\end{quest}

Indeed, since planar antipodal partial cubes are tope graphs of OMs of rank~$3$, see~\cite{Fuk-93}, any example for the above question has to be a non-planar antipodal partial cube of minimum degree $3$. It has been wondered whether the only non-planar cubic partial cube is the (antipodal) Desargues graph~\cite{Kla-08}, see the left of Figure~\ref{fig:desargues}. To our knowledge even the restriction to antipodal partial cubes remains open. For transitive cubic partial cubes it is known that the Desargues graphs is the only non-planar one, see~\cite{Mar-17}. On the other hand, it is open whether there are infinitely many non-planar partial cubes of minimum degree $3$.

\subsection{Upper bounds}

Bounding the minimum degree in a partial cubes $G$ from above by its rank is a generalization of Las Vergnas conjecture. As discussed in previous sections Las Vergnas conjecture is proved for OMs of rank at most $3$. In fact tope graphs of OMs of rank~$3$ are even $\Theta$-Las Vergnas, by Theorem~\ref{thm:simplical_on_theta} and the fact that they are Euclidean. We show that this property extends to general  antipodal partial cubes of rank~3. 

For this approach we introduce a couple of natural notions from~\cite{Kna-17}. A partial cube $G$ is called \emph{affine} if it is a halfspace $E_e^+$ of an antipodal partial  cube. The \emph{antipodes} $A(G)$ of an affine partial cube are those $u\in G$ such that there is $-u\in G$ such that the \emph{interval} $$[u,-u]=\{v\in G\mid \text{there is a shortest path from } u \text{ to } -u \text{ through } v\}$$ coincides with $G$. The antipodes of $G$ are exactly the vertices of $E_e^+$ incident to $E_e$ when $G$ is viewed as subgraph of $G'$. We need a auxiliary statement about the rank of affine partial cubes.

\begin{lem}\label{lem:rankantipodal}
 If an affine partial cube $G$ is a halfspace of an antipodal partial cube $G'$ of rank~$r$, then $G$ has rank at most $r-1$.
\end{lem}
\begin{proof}
 Suppose there is a sequence of contractions from $G$ to $Q_k$. Then the same sequence of contraction in $G'$ yields a minor $H$ with $Q_k$ as a halfspace. Since $H$ is antipodal, $H=Q_{k+1}$. 
\end{proof}
% 
% It was shown in~\cite{Kna-17} that affine partial cubes are closed under contractions. The following analyses the behavior of the set of antipodes under contraction.
% 
% \begin{lem}\label{lem:antipodesinaffines}
%  Let $G$ be affine with antipodes $A(G)$ and $E_e$ a $\Theta$-class. Then $A(\pi_e(G))=\pi_e(A(G))$.
% \end{lem}
% \begin{proof}
%  Let $u,v\in V(G)$ such that $\pi_e(u)=-\pi_e(v)$ in $\pi_e(G)$. Since $G$ is affine, by~\cite[Proposition 2.16]{Kna-17} there is an $x\in A(G)$ such that $[x,u]$ and $[v,-x]$ cross disjoint sets of $\Theta$-classes. Since $\pi_e(u)=-\pi_e(v)$ in $\pi_e(G)$, we have that $[u,v]$ crosses all classes of $G$ except possibly $E_e$. Thus, without loss of generality either $u=x$ or $u$ is incident with $E_e$ its neighbors with respect to $E_e$ is $x$ and $v=-x$. But then $\pi_e(x)=u$.
% \end{proof}

\begin{lem}\label{lem:LVaffinernk2}
 Every affine partial cube of rank at most $2$ has a vertex of degree at most $2$ among its antipodes.
\end{lem}
\begin{proof}
Since antipodal partial cubes of rank~$2$ are even cycles, it only remains to consider non-antipodal affine partial cubes. If $G$ is an affine partial cube of rank $2$, then by Lemma~\ref{lem:rankantipodal} it is a halfspace of an antipodal partial cube $G'$ of rank $3$. By~\cite[Theorem 7]{GKP70} since $G'$ is not a cycle its minimum degree is at least $3$. Hence if in $G$ there are vertices of degree $2$, then they are among its antipodes.

To show that $G$ has vertices of degree $2$, we use 
that by~\cite[Section 6.1]{CKP20}, that every rank $2$ partial cube can be augmented by adding vertices of degree at least $3$ to a COM $G''$ of rank $2$.

Finally, by Proposition~\ref{prop:rank2corner} the graph $G''$ has a corner $v$, which in particular is a vertex of degree at most $2$. By the above $v$ must already have been present in $G$. This concludes the proof.

% % So, let the affine partial cube $G$ be a minimal counterexample, i.e., all antipodes have degree at least $3$, but (since affine partial cubes are closed under contraction) every contraction destroys this property.% Note also that, since is not antipodal the set of antipodes of $G$ is contained in the one of $\pi_e(G)$ for any contration. 
%  
%  Thus, let $E_e$ be a $\Theta$-class. By minimality, in $\pi_e(G)$ there are two antipodes of degree $2$. Then by Lemma~\ref{lem:antipodesinaffines} there are two antipodal vertices $x,-x\in G$ such that in $\pi_e(G)$ they have degree $2$. Then, $x,-x$ are incident with $E_e$, call their neighbor with respect to $E_e$, $x'$ and $-x'$, respectively. Moreover, $x,-x$ have degree $3$ and their other two neighbors are also incident with $E_e$. Thus, also $x'$ and $-x'$ have at least two neighbors incident to $E_e$ and incident with the neighbors of $x$ and $-x$. Thus, contracting all other $\Theta$-classes yields a $Q_3$-minor -- contradiction.
\end{proof}

\begin{prop}\label{prop:strongLVrnk3}
 Let $G$ be an antipodal partial cube of rank~$3$ and $E_e$ a $\Theta$-class. There is a degree $3$ vertex incident to $E_e$.
\end{prop}
\begin{proof}
 Suppose that the claim is false. Let $G$ be a counterexample and $E_e$ a $\Theta$-class, such that all vertices incident to $E_e$ have degree at least $4$. Consider the contraction $G'=\pi_e(G)$ of $G$ and let $G'_2, G'_1$ be the antipodal expansion of $G'$ leading back to $G$. Since their preimage under $\pi_e$ has degree at least $4$, all vertices in $G'_1\cap G'_2$ have degree at least $3$. But $G'_1\cap G'_2$ are the antipodes of the affine partial cube $G'_1$. Moreover, $G'_1$ is of rank~$2$ by Lemma~\ref{lem:rankantipodal}. Thus, we have a contradiction with Lemma~\ref{lem:LVaffinernk2}.
\end{proof}

While we have already used several times, that even OMs of rank~$4$ are not $\Theta$-Las Vergnas, surprisingly enough Las Vergnas' conjecture could still hold for general antipodal partial cubes. We have verified it computationally up to isometric dimension $7$. See Table~\ref{tab:antipodal} for the numbers.

\begin{table}[hb]
\center
\begin{tabular}{r||c|c|c|c|c|c|c}
            $n$  & 2 & 3 & 4 &  5 &   6 & 7 & 8\\ \hline \hline
\text{antipodal} & 1 & 2 & 4 & 13 & 115 & 42257 & ?\\ \hline
\text{OM}        & 1 & 2 & 4 &  9 &  35 & 381 & 192449 \\\hline
\end{tabular}
\caption {Numbers of antipodal partial cubes and OMs of low isometric dimension. The latter can also be retrieved from \url{http://www.om.math.ethz.ch/}.}
\label{tab:antipodal}
\end{table} 

 Since already on isometric dimension $6$ there are 13488837 partial cubes, instead of filtering those of isometric dimension $7$ by antipodality, we filtered those of isometric dimension $6$ by affinity. There are $268615$ of them. We thus could create all antipodal partial cubes of dimension $7$ and count them and verify Las Vergnas' conjecture also for this set. We extend the  prolific Las Vergnas' conjecture to a much wider class.

\begin{quest}\label{quest:LV}
Does every antipodal partial cube of rank~$r$ have minimum degree at most $r$?
\end{quest}

% 
% \section{Apiculate graphs}
% 
% \begin{lem}
%  A vertex $v$ in an antipodal graph is a base of a poset such that all its neighbors haven joins if and only if $v$ is simplical.
% \end{lem}
% 
% 
% \begin{thm}
%  An antipodal partial cube is apiculate if and only if it is simplicial (regular?) and in $\topegraphs$.
% \end{thm}

\section{Conclusions and future work}
We have shown that Mandel OMs have the $\Theta$-Las Vergnas property, therefore disproving Mandel's conjecture. Finally, Las Vergnas' conjecture remains open and one of the most challenging open problems in OM theory. After computer experiments and a proof for rank~$3$, we dared to extend this question to general antipodal partial cubes, see Question~\ref{quest:LV}. Another strengthening of Las Vergnas' conjecture is the conjecture of Cordovil-Las Vergnas. We have verified it by computer for small examples and it holds for low rank in general. However, here we suspect the existence of a counter example at least in the setting of $\overline{\mathcal{G}}^{n,r}$.

Our second main contribution is the introduction of corner peelings for COMs and the proof of their existence in the realizable, rank~$2$, and hypercellular cases. A class that is a common generalization of the latter two is the class $\mathcal{S}_4$ of Pasch graphs. Do these graphs admit corner peelings? See Question~\ref{quest:pasch}.  

Let us close with two future directions of research that appear natural in the context of the objects discussed in this paper.

\subsection{Shellability}
In the context of an OM or AOM, a shelling is a special linear ordering of the vertices of its tope graph that yields a
\emph{recursive coatom ordering} of the full face lattice. See~\cite{bjvestwhzi-93} for the definitions. 
It thus, is natural to compare corner peelings and shellings.

It is known that AOMs and OMs are shellable, see~\cite{bjvestwhzi-93}. Shellability is defined for pure regular complexes thus the question of existence of shellings can be asked for all such COMs. However, the pure COM consisting of two 4-cycles joined in a vertex is not shellable.
\begin{quest}
Which COMs are shellable?
\end{quest}
A necessary condition might be $r$-connectedness, if a COM has rank $r$. Another useful hint might be the fact that an amalgamation procedure for COMs described in~\cite{Ban-18} is similar to the notion of constructibility, which is a weakening of shellability, see~\cite{hac-00}.

%We conjecture
%\begin{conj}
% COMs are shellable.
%\end{conj}

Corner peelings of LOPs are related to extendable shellability of the octahedron, see~\cite{Tra-04,Cha-18}. 
While OMs have corners, AOMs do not always have corners, as shown by the example of Figure~\ref{fig:tracyhall}. Hence, shellability does not imply the existence of a corner or a corner peeling. However, there still might be a connection:

\begin{quest}
 If a shellable COM $G$ has a corner peeling, can one find a shelling sequence that is a refinement of a corner peeling sequence of $G$?
\end{quest}

% 
% \comment{thougts:}
% \begin{itemize}
%  \item maybe all COMs are shellable in the same way OMs and AOMs are, i.e., taking a linear extension of the graph with a fixed base vertex $b$.
%  \item maybe corner peelings really correspond to extensions of partial shellings of a hypercube, see paper of Victor...they prove it there for LOPs.
%  \item I would now kinda guess, that corner peelings don't generally yield coarsened shellings...check at least for LOPs.
% \end{itemize}

\subsection{Murty's conjecture}

An important open problem in OMs is a generalization of the \emph{Sylvester-Gallai Theorem}, i.e., for every set of points in the plane that does not lie on a single line there is a line, that contains only two points. 

The corresponding conjecture in OMs can be found in Mandel's thesis~\cite{Man-82}, where it is attributed to Murty. In terms of OMs it reads:

\begin{conj}[Murty]\label{conj:murty}
 Every OM of rank~$r$ contains a convex subgraph that is the Cartesian product of an edge and an antipodal graph of rank~$r-2$.
\end{conj}

The realizable case of Murty's conjecture is shown by~\cite{Sha-79} and more generally holds for \emph{Mandel} OMs~\cite{Man-82}. Indeed, we suspect that along our strengthening of Mandel's theorem (Theorem~\ref{thm:simplical_on_theta}) a $\Theta$-version of Mandel's results can be proved:

\begin{conj}
Every $\Theta$-class in a Mandel OM of rank~$r$ is incident to a vertex of an antipodal graph that is the Cartesian product of an edge and an antipodal graph of rank~$r-2$.
\end{conj}

On the other hand it would be interesting to find OMs, that do not have this strengthened property. Still, Murty's conjecture in general seems out of reach. We propose a reasonable weaker statement to attack:

\begin{quest}\label{conj:cube}
 Does every OM of rank~$r$ contain a convex $Q_{\lceil\frac{r}{2}\rceil}$.
\end{quest}

\section{Acknowledgments}
We thank the referees for helpful comments and Arnaldo Mandel for providing us with a copy of his PhD thesis.
This work was supported by the Slovenian Research Agency (research core funding
No.\ P1-0297 and projects J1-9109, N1-0095, J1-1693). Tilen Marc wishes to express his gratitude to Institut français de Slovénie for supporting a visit in Marseille that started the present research. Kolja Knauer was furthermore supported by the ANR project DISTANCIA: ANR-17-CE40-0015 and by the Spanish \emph{Ministerio de Econom\'ia,
Industria y Competitividad} through grant RYC-2017-22701.
\bibliographystyle{my-siam}
{\footnotesize \bibliography{minors}}

\begin{thebibliography}{10}

\bibitem{Alb-16}
{\sc M.~{Albenque} and K.~{Knauer}}, {\em {Convexity in partial cubes: the hull
  number.}}, {Discrete Math.}, 339 (2016), pp.~866--876.

\bibitem{bab-01}
{\sc E.~Babson, L.~Finschi, and K.~Fukuda}, {\em Cocircuit graphs and efficient
  orientation reconstruction in oriented matroids}, Eur. J. Combin., 22 (2001),
  pp.~587--600.

\bibitem{Ban-96}
{\sc H.-J. {Bandelt} and V.~{Chepoi}}, {\em {Cellular bipartite graphs.}},
  {Eur. J. Comb.}, 17 (1996), pp.~121--134.

\bibitem{Ban-18}
{\sc H.-J. {Bandelt}, V.~{Chepoi}, and K.~{Knauer}}, {\em {COMs: complexes of
  oriented matroids.}}, {J. Comb. Theory Ser. A}, 156 (2018), pp.~195--237.

\bibitem{Bau-16}
{\sc A.~{Baum} and Y.~{Zhu}}, {\em {The axiomatization of affine oriented
  matroids reassessed.}}, {J. Geom.}, 109 (2018).

\bibitem{Bjo-90}
{\sc A.~{Bj\"orner}, P.~H. {Edelman}, and G.~M. {Ziegler}}, {\em {Hyperplane
  arrangements with a lattice of regions.}}, {Discrete Comput. Geom.}, 5
  (1990), pp.~263--288.

\bibitem{bjvestwhzi-93}
{\sc A.~Bj{\"o}rner, M.~Las~Vergnas, B.~Sturmfels, N.~White, and G.~M.
  Ziegler}, {\em Oriented matroids}, vol.~46 of Encyclopedia of Mathematics and
  its Applications, Cambridge University Press, Cambridge, second~ed., 1999.

\bibitem{Bok-01}
{\sc J.~{Bokowski} and H.~{Rohlfs}}, {\em {On a mutation problem for oriented
  matroids.}}, {Eur. J. Comb.}, 22 (2001), pp.~617--626.

\bibitem{Cha-18}
{\sc J.~Chalopin, V.~Chepoi, S.~Moran, and M.~K. Warmuth}, {\em {Unlabeled
  Sample Compression Schemes and Corner Peelings for Ample and Maximum
  Classes}}, in 46th International Colloquium on Automata, Languages, and
  Programming (ICALP 2019), vol.~132, 2019, pp.~34:1--34:15.

\bibitem{Che-94}
{\sc V.~{Chepoi}}, {\em {Separation of two convex sets in convexity
  structures.}}, {J. Geom.}, 50 (1994), pp.~30--51.

\bibitem{Che-16}
{\sc V.~Chepoi, K.~Knauer, and T.~Marc}, {\em Hypercellular graphs: Partial
  cubes without ${Q}_3^-$ as partial cube minor}, Discrete Math., 343 (2020),
  p.~111678.

\bibitem{CKP20}
{\sc V.~Chepoi, K.~Knauer, and M.~Philibert}, {\em Two-dimensional partial
  cubes}, Electron. J. Comb., 27 (2020), pp.~research paper p3.29, 40.

\bibitem{CKP22}
{\sc V.~Chepoi, K.~Knauer, and M.~Philibert}, {\em Ample completions of
  oriented matroids and complexes of uniform oriented matroids}, SIAM J.
  Discrete Math., 36 (2022), pp.~509--535.

\bibitem{cordovil2000cocircuit}
{\sc R.~Cordovil, K.~Fukuda, and A.~G. de~Oliveira}, {\em On the cocircuit
  graph of an oriented matroid}, Discrete Comput. Geom., 24 (2000),
  pp.~257--266.

\bibitem{daS-95}
{\sc I.~P.~F. da~Silva}, {\em Axioms for maximal vectors of an oriented
  matroid: a combinatorial characterization of the regions determined by an
  arrangement of pseudohyperplanes}, Eur. J. Combin., 16 (1995), pp.~125--145.

\bibitem{DKR21}
{\sc G.~Dorpalen-Barry, J.~S. Kim, and V.~Reiner}, {\em {Whitney Numbers for
  Poset Cones.}}, {Order},  (2021).

\bibitem{omcubic}
{\sc S.~Felsner, R.~G\'omez, K.~Knauer, J.~J. Montellano-Ballesteros, and
  R.~Strausz}, {\em {Cubic time recognition of cocircuit graphs of uniform
  oriented matroids.}}, {Eur. J. Comb.}, 32 (2011), pp.~60--66.

\bibitem{fo-la-78}
{\sc J.~Folkman and J.~Lawrence}, {\em Oriented matroids}, J. Combinatorial
  Theory Ser. B, 25 (1978), pp.~199--236.

\bibitem{Fuk-82}
{\sc K.~Fukuda}, {\em Oriented matroid programming}, 1982.
\newblock PhD Thesis, University of Waterloo.

\bibitem{Fuk-93}
{\sc K.~Fukuda and K.~Handa}, {\em Antipodal graphs and oriented matroids},
  Discrete Math., 111 (1993), pp.~245--256.

\bibitem{GKP70}
{\sc F.~Glivjak, A.~Kotzig, and J.~Plesn{\'{\i}}k}, {\em Remark on the graphs
  with a central symmetry}, Monatsh. Math., 74 (1970), pp.~302--307.

\bibitem{hac-00}
{\sc M.~Hachimori}, {\em Combinatorics of constructible complexes}, 2000.
\newblock PhD thesis, University of Tokyo.

\bibitem{ham-11}
{\sc R.~Hammack, W.~Imrich, and S.~Klav{\v{z}}ar}, {\em Handbook of product
  graphs}, Discrete Mathematics and its Applications (Boca Raton), CRC Press,
  Boca Raton, FL, 2011.

\bibitem{ha-90}
{\sc K.~Handa}, {\em A characterization of oriented matroids in terms of
  topes}, Eur. J. Combin., 11 (1990), pp.~41--45.

\bibitem{Hoch-19}
{\sc W.~{Hochst\"attler} and V.~{Welker}}, {\em {The Varchenko determinant for
  oriented matroids.}}, {Math. Z.}, 293 (2019), pp.~1415--1430.

\bibitem{imrich1998convexity}
{\sc W.~Imrich and S.~Klav{\v{z}}ar}, {\em A convexity lemma and expansion
  procedures for bipartite graphs}, Eur. J. Comb., 19 (1998), pp.~677--685.

\bibitem{itskov2018hyperplane}
{\sc V.~Itskov, A.~Kunin, and Z.~Rosen}, {\em Hyperplane neural codes and the
  polar complex}, in Topological data analysis. Proceedings of the Abel
  symposium 2018, Geiranger, Norway, June 4--8, 2018, Cham: Springer, 2020,
  pp.~343--369.

\bibitem{junttila2007engineering}
{\sc T.~Junttila and P.~Kaski}, {\em Engineering an efficient canonical
  labeling tool for large and sparse graphs}, in 2007 Proceedings of the Ninth
  Workshop on Algorithm Engineering and Experiments (ALENEX), SIAM, 2007,
  pp.~135--149.

\bibitem{Kla-08}
{\sc S.~Klavzar}, {\em Hunting for cubic partial cubes}, Convexity in Discrete
  Structures, 5 (2008), pp.~87--95.

\bibitem{Kla-12}
{\sc S.~Klav{\v{z}}ar and S.~Shpectorov}, {\em Convex excess in partial cubes},
  J. Graph Theory, 69 (2012), pp.~356--369.

\bibitem{Kna-17}
{\sc K.~Knauer and T.~Marc}, {\em On tope graphs of complexes of oriented
  matroids}, Discrete Comput. Geom.,  (2019), pp.~377--417.

\bibitem{omax}
{\sc K.~Knauer, J.~J. Montellano-Ballesteros, and R.~Strausz}, {\em A
  graph-theoretical axiomatization of oriented matroids}, Eur. J. Combin., 35
  (2014), pp.~388--391.

\bibitem{alex2020oriented}
{\sc A.~Kunin, C.~Lienkaemper, and Z.~Rosen}, {\em Oriented matroids and
  combinatorial neural codes}, 2020.

\bibitem{lv-80}
{\sc M.~{Las Vergnas}}, {\em {Convexity in oriented matroids.}}, {J. Comb.
  Theory Ser. B}, 29 (1980), pp.~231--243.

\bibitem{Law-83}
{\sc J.~Lawrence}, {\em Lopsided sets and orthant-intersection by convex sets},
  Pacific J. Math., 104 (1983), pp.~155--173.

\bibitem{Lev-26}
{\sc F.~Levi}, {\em Die {T}eilung der projektiven {E}bene durch {G}erade oder
  {P}seudogerade}, Ber. Math.-Phys. Kl. S{\"a}chs. Akad. Wiss, 78 (1926),
  pp.~256--267.

\bibitem{Man-82}
{\sc A.~Mandel}, {\em Topology of oriented matroids}, 1982.
\newblock PhD Thesis, University of Waterloo.

\bibitem{Mar-17}
{\sc T.~{Marc}}, {\em {Classification of vertex-transitive cubic partial
  cubes.}}, {J. Graph Theory}, 86 (2017), pp.~406--421.

\bibitem{mar-18}
{\sc T.~Marc}, {\em Cycling in hypercubes}, PhD thesis, Univ. of Ljubljana,
  Faculty of Mathematics and Physics, 2018.

\bibitem{margolis2015cell}
{\sc S.~Margolis, F.~V. Saliola, and B.~Steinberg}, {\em Cell complexes, poset
  topology and the representation theory of algebras arising in algebraic
  combinatorics and discrete geometry}, vol.~1345 of Mem. Am. Math. Soc.,
  Providence, RI: American Mathematical Society (AMS), 2021.

\bibitem{mon-06}
{\sc J.~J. Montellano-Ballesteros and R.~Strausz}, {\em A characterization of
  cocircuit graphs of uniform oriented matroids}, J. Combin. Theory Ser. B, 96
  (2006), pp.~445--454.

\bibitem{Ovc-11}
{\sc S.~{Ovchinnikov}}, {\em {Graphs and cubes.}}, Berlin: Springer, 2011.

\bibitem{padrol2021sweeps}
{\sc A.~Padrol and E.~Philippe}, {\em Sweeps, polytopes, oriented matroids, and
  allowable graphs of permutations}, 2021.

\bibitem{polat2018some}
{\sc N.~Polat}, {\em On some properties of antipodal partial cubes}, Discuss.
  Math., Graph Theory,  (2018).

\bibitem{Pol-19}
{\sc N.~{Polat}}, {\em {On some characterizations of antipodal partial
  cubes.}}, {Discuss. Math., Graph Theory}, 39 (2019), pp.~439--453.

\bibitem{R20}
{\sc H.~{Randriamaro}}, {\em {The Varchenko determinant for apartments}},
  {Result. Math.}, 75 (2020), p.~17.
\newblock Id/No 86.

\bibitem{Ric-93}
{\sc J.~{Richter-Gebert}}, {\em {Oriented matroids with few mutations.}},
  {Discrete Comput. Geom.}, 10 (1993), pp.~251--269.

\bibitem{Rin-56}
{\sc G.~{Ringel}}, {\em {Teilungen der Ebenen durch Geraden oder topologische
  Geraden.}}, {Math. Z.}, 64 (1956), pp.~79--102.

\bibitem{Rin-57}
{\sc G.~{Ringel}}, {\em {\"Uber Geraden in allgemeiner Lage.}}, {Elem. Math.},
  12 (1957), pp.~75--82.

\bibitem{Rou-88}
{\sc J.-P. {Roudneff} and B.~{Sturmfels}}, {\em {Simplicial cells in
  arrangements and mutations of oriented matroids.}}, {Geom. Dedicata}, 27
  (1988), pp.~153--170.

\bibitem{San-02}
{\sc F.~{Santos}}, {\em {Triangulations of oriented matroids.}}, vol.~741,
  American Mathematical Society (AMS), 2002.

\bibitem{Sha-79}
{\sc R.~W. {Shannon}}, {\em {Simplicial cells in arrangements of
  hyperplanes.}}, {Geom. Dedicata}, 8 (1979), pp.~179--187.

\bibitem{Tra-04}
{\sc H.~Tracy~Hall}, {\em Counterexamples in Discrete Geometry}, PhD Thesis,
  University Of California, Berkeley, 2004.

\bibitem{VC-15}
{\sc V.~N. {Vapnik} and A.~Y. {Chervonenkis}}, {\em {On the uniform convergence
  of relative frequencies of events to their probabilities.}}, in {Measures of
  complexity}, Springer, 2015, pp.~11--30.

\end{thebibliography}

\end{document}